\documentclass[10pt,]{article} 
\usepackage[numbers]{natbib}
\usepackage[english]{babel}
\usepackage{bbm}
\usepackage{amsmath}         		
\usepackage{amssymb}
\usepackage{amsthm}
\usepackage{color}
\usepackage{multirow}
\usepackage{empheq}
\usepackage{graphicx}
\usepackage{tabularx}
\usepackage{array}
\usepackage{thmtools,thm-restate}
\usepackage{url}

\usepackage[
 letterpaper,
 textheight=9in,
 textwidth=6.5in,
 top=1in
 ]{geometry}
  
\newcommand{\set}[1]{\left\{ #1 \right\}}
\newcommand{\inner}[2]{\left\langle #1,\, #2 \right\rangle}
\newcommand{\norm}[1]{\left\| #1 \right\|}

\DeclareMathOperator*{\argmin}{arg\,min}
\newcommand{\abs}[1]{\left| #1 \right|}
\newcommand{\spn}[1]{\text{span}\set{ #1 }}
\newcommand{\prox}[0]{\text{prox}}
\newcommand{\R}{\mathbb{R}}
\newcommand{\E}[1]{{\mathbb{E}\left[{#1}\right]}}

\newcommand{\normalized}[1]{\frac{#1}{\norm{#1}}}
\newcommand{\parens}[1]{\left( #1 \right)}

\newtheorem{theorem}{Theorem}
\newtheorem{lemma}{Lemma}
\newtheorem{corollary}{Corollary}

\newcommand{\removed}[1]{}

\title{\rule{6.5in}{2pt}\\Tight Complexity Bounds for Optimizing Composite Objectives\\\rule[2mm]{6.5in}{0.5pt}}
\author{}
\date{}
\begin{document}
\maketitle
\vspace{-21mm}
\textbf{Blake Woodworth} \hfill \url{blake@ttic.edu}\\
\textbf{Nathan Srebro} \hfill \url{nati@ttic.edu}\\
{\small Toyota Technological Institute at Chicago, Chicago, IL 60637, USA}

\begin{abstract}
We provide tight upper and lower bounds on the complexity of
minimizing the average of $m$ convex functions using gradient and
prox oracles of the component functions.  We show a significant gap
between the complexity of deterministic vs randomized optimization.
For smooth functions, we show that accelerated gradient descent (AGD)
and an accelerated variant of SVRG are optimal in the deterministic and
randomized settings respectively, and that a gradient oracle is
sufficient for the optimal rate.  For non-smooth functions, having access to prox oracles
reduces the complexity and we present optimal methods based on
smoothing that improve over methods using just gradient
accesses.
\end{abstract}

\section{Introduction}
We consider minimizing the average of $m \geq 2$ convex functions:
\begin{equation} \label{eq:main}
\min_{x \in \mathcal{X}}\set{ F(x) := \frac{1}{m} \sum_{i=1}^m f_i(x) }
\end{equation}
where $\mathcal{X}\subseteq \mathbb{R}^d$ is a closed, convex set, and where the algorithm is given
access to the following gradient (or subgradient in the case of non-smooth functions) and prox oracle for the components:
\begin{equation}
h_F(x,i,\beta) = \left[f_i(x),\ \nabla f_i(x),\ \prox_{f_i}(x,\beta)\right]
\end{equation}
where
\begin{equation} \label{eq:prox}
\prox_{f_i}(x,\beta) = \argmin_{u \in \mathcal{X}} \set{ f_i(u) + \frac{\beta}{2}\norm{x-u}^2 }
\end{equation}

A natural question is how to leverage the prox oracle, and how
much benefit it provides over gradient access alone.  The
prox oracle is potentially much more powerful, as it provides global, rather then local, information about the
function.  For example, for a single function ($m=1$), one prox oracle call (with
$\beta=0$) is sufficient for exact optimization.  Several methods have
recently been suggested for optimizing a sum or average of several
functions using prox accesses to each component, both in the
distributed setting where each components might be handled on a different
machine (e.g.~ADMM \citep{ADMM}, DANE \citep{DANE}, DISCO \citep{DISCO}) or for
functions that can be decomposed into several ``easy'' parts
(e.g.~PRISMA \citep{PRISMA}).  But as far as we are
aware, no meaningful lower bound was previously known on the number of
prox oracle accesses required even for the average of two functions ($m=2$).

The optimization of composite objectives of the form \eqref{eq:main}
has also been extensively studied in the context of minimizing
empirical risk over $m$ samples.  Recently, stochastic methods such as SDCA \citep{SDCA}, SAG
\citep{SAG}, SVRG \citep{SVRG}, and other variants, have been presented
which leverage the finite nature of the problem to reduce the variance
in stochastic gradient estimates and obtain guarantees that dominate both
batch and stochastic gradient descent.  As methods with improved
complexity, such as accelerated SDCA \citep{ASDCA}, accelerated SVRG, and \textsc{Katyusha} \citep{Katyusha}, have been presented,
researchers have also tried to obtain lower bounds on the best possible complexity in this settings---but 
as we survey below, these have not been satisfactory so far.

In this paper, after briefly surveying methods for smooth, composite optimization,
we present methods for optimizing non-smooth composite objectives, 
which show that prox oracle access can indeed be leveraged to improve 
over methods using merely subgradient access (see Section \ref{sec:LipUppers}).  We then
turn to studying lower bounds. We consider algorithms that
access the objective $F$ only through the oracle $h_F$ and provide
lower bounds on the number of such oracle accesses (and thus the
runtime) required to find $\epsilon$-suboptimal solutions.  We
consider optimizing both Lipschitz (non-smooth) functions and smooth
functions, and guarantees that do and do not depend on strong
convexity, distinguishing between deterministic optimization
algorithms and randomized algorithms.  Our upper and lower bounds are
summarized in Table \ref{tab:main}.

\begin{table}
\center
\setlength{\tabcolsep}{1.2mm}
\begin{tabularx}{\textwidth}{| c | c | >{\centering}m{38mm} | >{\centering}m{32mm} | *2{>{\centering\arraybackslash}X|}}
\cline{3-6}
\multicolumn{2}{c|}{}
&\multicolumn{2}{c|}{$L$-Lipschitz}
&\multicolumn{2}{c|}{$\gamma$-Smooth} \\ \cline{3-6}
\multicolumn{2}{c|}{}& Convex, & $\lambda$-Strongly & Convex, & $\lambda$-Strongly \vspace{-4mm} \\ 
\multicolumn{2}{c|}{}& $\norm{x}\leq B$ & Convex & $\norm{x}\leq B$ & Convex \\ 
\hline
%
\multirow{4}{*}{\rotatebox[origin=c]{90}{Deterministic}} & \multirow{2}{*}{\rotatebox[origin=c]{90}{Upper}}
&$\frac{mLB}{\epsilon}$
&$\frac{mL}{\sqrt{\lambda\epsilon}}$
&$m\sqrt{\frac{\gamma B^2}{\epsilon}}$
&$m\sqrt{\frac{\gamma}{\lambda}}\log\frac{\epsilon_0}{\epsilon}$
\rule[0pt]{0pt}{14pt} 
\\&
& {\scriptsize \textbf{(Section 3)}}
& {\scriptsize \textbf{(Section 3)}}
& {\scriptsize (AGD)}
& {\scriptsize (AGD)}
\rule[-6pt]{0pt}{6pt} 
\\\cline{2-6}
%
& \multirow{2}{*}{\rotatebox[origin=c]{90}{Lower}}
&$\frac{mLB}{\epsilon}$
&$ \frac{mL}{\sqrt{\lambda\epsilon}}$
&$m\sqrt{\frac{\gamma B^2}{\epsilon}}$
&$m\sqrt{\frac{\gamma}{\lambda}}\log\frac{\epsilon_0}{\epsilon}$ 
\rule[0pt]{0pt}{14pt}
\\&
& {\scriptsize \textbf{(Section 4)}}
& {\scriptsize \textbf{(Section 4)}}
& {\scriptsize \textbf{(Section 4)}}
& {\scriptsize \textbf{(Section 4)}}
\rule[-6.5pt]{0pt}{6.5pt}
\\\hline
%
\multirow{4}{*}{\rotatebox[origin=c]{90}{$\,$Randomized}} & \multirow{2}{*}{\rotatebox[origin=c]{90}{Upper}}
&{\small $\tfrac{L^2B^2}{\epsilon^2}\!\!\wedge\!\!\left(\!m\log\!\tfrac{1}{\epsilon} \!+\! \frac{\sqrt{m}LB}{\epsilon}\!\right)$} 
&{\small $\tfrac{L^2}{\lambda\epsilon}\!\!\wedge\!\!\left(\!m\log\!\frac{1}{\epsilon} \!+\! \frac{\sqrt{m}L}{\sqrt{\lambda\epsilon}}\!\right)$}
&{\small $m \log\!\frac{\epsilon_0}{\epsilon}\!+\!\sqrt{\!\frac{m\gamma B^2}{\epsilon}}$}
&{\small $\left(m \!+\! \sqrt{\!\frac{m\gamma}{\lambda}} \right)\!\log\!\tfrac{\epsilon_0}{\epsilon}$} 
\\&
& {\scriptsize (SGD, A-SVRG)}
& {\scriptsize (SGD, A-SVRG)}
& {\scriptsize (A-SVRG)}
& {\scriptsize (A-SVRG)}
\rule[-6pt]{0pt}{6pt} 
\\\cline{2-6}
%
&\multirow{2}{*}{\rotatebox[origin=c]{90}{Lower}}
&{\small $\tfrac{L^2B^2}{\epsilon^2} \!\wedge\! \left(\!m \!+\!\frac{\sqrt{m}LB}{\epsilon}\!\right)$}
&{\small $\tfrac{L^2}{\lambda\epsilon} \!\wedge\! \left(\!m \!+\! \frac{\sqrt{m}L}{\sqrt{\lambda\epsilon}}\!\right)$}
&{\small $m + \sqrt{\frac{m\gamma B^2}{\epsilon}}$}
&{\small $m + \sqrt{\frac{m\gamma}{\lambda}}\log\frac{\epsilon_0}{\epsilon}$}
\rule[-3pt]{0pt}{15pt} 
\\&
& {\scriptsize \textbf{(Section 5)}}
& {\scriptsize \textbf{(Section 5)}}
& {\scriptsize \textbf{(Section 5)}}
& {\scriptsize \textbf{(Section 5)}}
\rule[-6.5pt]{0pt}{6.5pt} 
\\\hline
\end{tabularx}
\caption{\small Upper and lower bounds on the number
  of grad-and-prox oracle accesses needed to find
  $\epsilon$-suboptimal solutions for each function class. These are exact up to constant factors except for the lower bounds for smooth and strongly convex functions, which hide extra $\log\lambda/\gamma$ and $\log \sqrt{m\lambda/\gamma}$ factors for deterministic and randomized algorithms. Here, $\epsilon_0$ is the suboptimality of the point $0$.}\label{tab:main}
\vspace{-3mm}
\end{table}

As shown in the table, we provide matching upper and lower bounds (up
to a log factor) for all function and algorithm classes.  In particular, our bounds
establish the optimality (up to log factors) of accelerated SDCA,
SVRG, and SAG for randomized finite-sum optimization, and
also the optimality of our deterministic smoothing algorithms for
non-smooth composite optimization.

\paragraph{On the power of gradient vs prox oracles} For non-smooth
functions, we show that having access to prox oracles for the
components can reduce the polynomial dependence on $\epsilon$ from
$1/\epsilon^2$ to $1/\epsilon$, or from $1/(\lambda \epsilon)$ to
$1/\sqrt{\lambda \epsilon}$ for $\lambda$-strongly convex functions.
However, all of the optimal complexities for smooth functions can be
attained with only component gradient access using accelerated gradient descent (AGD) or accelerated
SVRG.  Thus the worst-case complexity cannot be improved
(at least not significantly) by using the more powerful prox oracle.

\paragraph{On the power of randomization} We establish a significant
gap between deterministic and randomized algorithms for finite-sum
problems. Namely, the dependence on the number of components must be
linear in $m$ for any deterministic algorithm, but can be reduced to
$\sqrt{m}$ (in the typically significant term) using randomization.
We emphasize that the randomization here is only in the
algorithm---not in the oracle.  We always assume the oracle returns an
exact answer (for the requested component) and is {\em not} a
stochastic oracle.  The distinction is that the {\em algorithm} is
allowed to flip coins in deciding what operations and queries to perform but the oracle must
return an exact answer to that query (of course, the algorithm could simulate a stochastic oracle).

\paragraph{Prior Lower Bounds} 
Several authors recently presented lower bounds for optimizing
\eqref{eq:main} in the smooth and strongly convex setting using
component gradients.  \citet{AgarwalBottou} presented a lower bound of
$\Omega\left( m + \sqrt{\tfrac{m\gamma}{\lambda}}
  \log\tfrac{1}{\epsilon} \right)$.  However, their bound is valid only
for {\em deterministic} algorithms (thus not including SDCA, SVRG,
SAG, etc.)---we not only consider randomized algorithms, but also show
a much higher lower bound for deterministic algorithms
(i.e.~the bound of \citeauthor{AgarwalBottou} is loose).
Improving upon this, \citet{Lan} shows a
similar lower bound for a restricted class of randomized algorithms: the
algorithm must select which component to query for a gradient by
drawing an index from a fixed distribution, but the algorithm must otherwise be
deterministic in how it uses the gradients, and its iterates must lie
in the span of the gradients it has received.  This restricted class
includes SAG, but not SVRG nor perhaps other realistic attempts
at improving over these. Furthermore, both bounds allow only
gradient accesses, not prox computations. Thus
SDCA, which requires prox accesses, and potential variants are
not covered by such lower bounds. We prove as similar lower bound to Lan's,
 but our analysis is much more general and applies to
\emph{any} randomized algorithm, making {\em any} sequence of queries to a gradient {\em and} prox oracle, and
without assuming that iterates lie in the span of previous
responses. In addition to smooth
functions, we also provide lower bounds for non-smooth problems which
were not considered by these previous attempts.  Another recent
observation \cite{ShaiSlides} was that with access only to random
component subgradients {\em without} knowing the component's identity, an
algorithm must make $\Omega(m^2)$ queries to optimize well. This shows how
relatively subtle changes in the oracle can have a dramatic effect on
the complexity of the problem. Since the oracle we consider is quite powerful,
our lower bounds cover a very broad family of algorithms, including SAG, SVRG, and SDCA.

Our deterministic lower bounds are inspired by a lower
bound on the number of rounds of communication required for
optimization when each $f_i$ is held by a different machine and when
iterates lie in the span of certain permitted calculations
\citep{Ohad}.  Our construction for $m=2$ is similar to theirs (though in a
different setting), but their analysis considers neither scaling with $m$
(which has a different role in their setting) nor randomization.

\paragraph{Notation and Definitions}
We use $\norm{\cdot}$ to denote the standard Euclidean norm on
$\mathbb{R}^d$.  We say that a function $f$ is $L$-Lipschitz
continuous on $\mathcal{X}$ if $\forall x,y \in \mathcal{X}\ 
\abs{f(x) - f(x)} \leq L\norm{x - y}$; $\gamma$-smooth on
$\mathcal{X}$ if it is differentiable and its gradient is
$\gamma$-Lipschitz on $\mathcal{X}$; and $\lambda$-strongly
convex on $\mathcal{X}$ if $\forall x,y \in \mathcal{X}\quad f_i(y)
\geq f_i(x) + \inner{\nabla f_i(x)}{ y-x} +
\frac{\lambda}{2}\norm{x-y}^2$.  We consider optimizing \eqref{eq:main}
 under four combinations of assumptions: each
component $f_i$ is either $L$-Lipschitz or $\gamma$-smooth, and
either $F(x)$ is $\lambda$-strongly convex or its domain is
bounded, $\mathcal{X}\subseteq\left\{x:\norm{x}\leq B\right\}$.

\section{Optimizing Smooth Sums \label{sec:SmoothUpper}}

We briefly review the best known methods for optimizing
\eqref{eq:main} when the components are $\gamma$-smooth, yielding the
upper bounds on the right half of Table \ref{tab:main}.  These upper
bounds can be obtained using only component gradient access, without
need for the prox oracle.

We can obtain exact gradients of $F(x)$ by computing all $m$ component
gradients $\nabla f_i(x)$.
Running accelerated gradient descent (AGD) \citep{NesterovAGD} on
$F(x)$ using these exact gradients achieves the upper complexity bounds for deterministic
algorithms and smooth problems (see
Table \ref{tab:main}).

SAG \cite{SAG}, SVRG \cite{SVRG} and related methods use randomization to
sample components, but also leverage the finite nature of the objective to control the
variance of the gradient estimator used.  Accelerating these methods
using the Catalyst framework \cite{GeneralAcceleration} ensures that for
$\lambda$-strongly convex objectives we have $\E{F(x^{(k)})-F(x^*)} <
\epsilon$ after $k=\mathcal{O}\left(\left(m +
    \sqrt{\tfrac{m\gamma}{\lambda}}\right) \log^2
  \tfrac{\epsilon_0}{\epsilon} \right)$ iterations, where $F(0) - F(x^*) =
\epsilon_0$.  \textsc{Katyusha} \cite{Katyusha} is a more direct approach to
accelerating SVRG which avoids extraneous log-factors, yielding the
complexity $k=\mathcal{O}\left(\left(m +
    \sqrt{\tfrac{m\gamma}{\lambda}}\right)
  \log\tfrac{\epsilon_0}{\epsilon} \right)$ indicated in Table \ref{tab:main}.

When $F$ is not strongly convex, adding a regularizer to the objective
and instead optimizing
$F_{\lambda}(x)=F(x)+\frac{\lambda}{2}\norm{x}^2$ with
$\lambda=\epsilon/B^2$ results in an oracle complexity of
$\mathcal{O}\left(\left(m + \sqrt{\tfrac{m\gamma B^2}{\epsilon}}\right)
  \log\tfrac{\epsilon_0}{\epsilon} \right)$.  The log-factor in the
second term can be removed using the more delicate reduction of
\citet{HazanReduction}, which involves optimizing $F_\lambda(x)$ for
progressively smaller values of $\lambda$, yielding the upper bound in
the table.

\textsc{Katyusha} and Catalyst-accelerated SAG or SVRG use only
gradients of the components.  Accelerated SDCA \cite{ASDCA} achieves
a similar complexity using gradient and prox oracle access.

\section{Leveraging Prox Oracles for Lipschitz Sums \label{sec:LipUppers}}

In this section, we present algorithms for leveraging the prox
oracle to minimize \eqref{eq:main} when each component is
$L$-Lipschitz.  This will be done by using the prox oracle to
``smooth'' each component, and optimizing the new, smooth sum which
approximates the original problem.  This idea was used in order to
apply \textsc{Katyusha} \cite{Katyusha} and accelerated SDCA \cite{ASDCA} to
non-smooth objectives.  We are not aware of a previous explicit
presentation of the AGD-based deterministic algorithm, which achieves
the deterministic upper complexity indicated in Table \ref{tab:main}.

The key is using a prox oracle to obtain gradients of the
$\beta$-Moreau envelope of a non-smooth function, $f$, defined as:
\begin{equation}\label{eq:moreau}
f^{(\beta)}(x) = \inf_{u\in\mathcal{X}} f(u) + \frac{\beta}{2}\norm{x-u}^2
\end{equation}

\begin{lemma}[{\citep[Lemma 2.2]{PRISMA}},
{\citep[Proposition 12.29]{BauschkeCombettes}}, 
following \citep{Nesterov05}]\label{lem:prox}
Let $f$ be convex and $L$-Lipschitz continuous.  For any $\beta > 0$,
\begin{enumerate}
\item $f^{(\beta)}$ is $\beta$-smooth
\item $\nabla (f^{(\beta)})(x) = \beta (x - \prox_{f}(x,\beta))$
\item $f^{(\beta)}(x) \leq f(x) \leq f^{(\beta)}(x) + \frac{L^2}{2\beta}$
\end{enumerate}
\end{lemma}
Consequently, we can consider the smoothed problem
\begin{equation} \label{eq:tildeF}
\min_{x \in \mathcal{X}}\set{ \tilde{F}^{(\beta)}(x) := \frac{1}{m} \sum_{i=1}^m f^{(\beta)}_i(x) }.
\end{equation}
While $\tilde{F}^{(\beta)}$ is \emph{not}, in general, the
$\beta$-Moreau envelope of $F$, it is $\beta$-smooth, we can calculate
the gradient of its components using the oracle $h_F$, and
$\tilde{F}^{(\beta)}(x) \leq F(x) \leq \tilde{F}^{(\beta)}(x) +
\frac{L^2}{2\beta}$.  Thus, to obtain an $\epsilon$-suboptimal
solution to \eqref{eq:main} using $h_F$, we set $\beta=L^2/\epsilon$
and apply any algorithm which can optimize \eqref{eq:tildeF} using
gradients of the $L^2/\epsilon$-smooth components, to within
$\epsilon/2$ accuracy. With the rates presented in Section
\ref{sec:SmoothUpper}, using AGD on \eqref{eq:tildeF} yields a
complexity of $\mathcal{O}\left(\frac{mLB}{\epsilon}\right)$ in the
deterministic setting.  When the functions are $\lambda$-strongly
convex, smoothing with a fixed $\beta$ results in a spurious
log-factor. To avoid this, we again apply the reduction of
\citet{HazanReduction}, this time optimizing $\tilde{F}^{(\beta)}$ for
increasingly large values of $\beta$. This leads to the upper bound of
$\mathcal{O}\left(\frac{mL}{\sqrt{\lambda\epsilon}}\right)$ when used
with AGD (see Appendix \ref{appendix:LipUpper} for details).

Similarly, we can apply an accelerated randomized algorithm (such as
\textsc{Katyusha}) to the smooth problem $\tilde{F}^{(\beta)}$ to obtain complexities of
$\mathcal{O}\left(m\log\frac{\epsilon_0}{\epsilon} +
  \frac{\sqrt{m}LB}{\epsilon}\right)$ and
$\mathcal{O}\left(m\log\frac{\epsilon_0}{\epsilon} +
  \frac{\sqrt{m}L}{\sqrt{\lambda\epsilon}}\right)$---this matches the
presentation of \citet{Katyusha} and is similar to that of \citet{ASDCA}.

Finally, if $m > L^2 B^2/\epsilon^2$ or $m>L^2/(\lambda \epsilon)$,
stochastic gradient descent is a better randomized alternative,
yielding complexities of $\mathcal{O}(L^2 B^2/\epsilon^2)$ or
$\mathcal{O}(L^2/(\lambda \epsilon))$.


\section{Lower Bounds for Deterministic Algorithms}

We now turn to establishing lower bounds on the oracle complexity of
optimizing \eqref{eq:main}.  We first consider only
deterministic optimization algorithms.  What we would like to show is
that for any deterministic optimization algorithm we can construct a
``hard'' function for which the algorithm cannot find an
$\epsilon$-suboptimal solution until it has made many oracle accesses.  Since
the algorithm is deterministic, we can construct such a function by
simulating the (deterministic) behavior of the algorithm.  This can
be viewed as a game, where an adversary controls the oracle being
used by the algorithm. At each iteration the algorithm queries the
oracle with some triplet $(x,i,\beta)$ and the adversary responds with
an answer. This answer must be consistent with all previous answers, but the adversary 
ensures it is also consistent with a composite function $F$ that the
algorithm is far from optimizing. The ``hard'' function is then
gradually defined in terms of the behavior of the optimization
algorithm.

To help us formulate our constructions, we define a ``round'' of
queries as a series of queries in which $\lceil
\frac{m}{2} \rceil$ distinct functions $f_i$ are queried. The first
round begins with the first query and continues until exactly $\lceil
\frac{m}{2} \rceil$ unique functions have been queried. The second
round begins with the next query, and continues until exactly $\lceil
\frac{m}{2} \rceil$ more distinct components have been queried in the
second round, and so on until the algorithm terminates. This definition 
is useful for analysis but requires no assumptions about
the algorithm's querying strategy.

\subsection{Non-Smooth Components}

We begin by presenting a lower bound for deterministic optimization of
\eqref{eq:main} when each component $f_i$ is
convex and $L$-Lipschitz continuous, but is not necessarily strongly convex,
on the domain $\mathcal{X}=\left\{ x :
  \norm{x}\leq B \right\}$.  Without loss of generality, we can
consider $L=B=1$.  We will construct functions of the following form:
\begin{equation}\label{eq:DetLipF}
f_i(x) = \frac{1}{\sqrt{2}}\abs{b - \inner{x}{v_0}} + \frac{1}{2\sqrt{k}} \sum_{r=1}^{k} \delta_{i,r} \abs{\inner{x}{v_{r-1}} - \inner{x}{v_{r}}}.
\end{equation}
where $k=\lfloor \tfrac{1}{12\epsilon} \rfloor$, $b=\tfrac{1}{\sqrt{k+1}}$, and $\set{v_r}$ is an orthonormal set of vectors in $\R^d$ chosen 
according to the behavior of the algorithm such that $v_r$ is orthogonal
to all points at which the algorithm queries $h_F$ before round $r$, and where $\delta_{i,r}$ are indicators chosen so that
$\delta_{i,r}=1$ if the algorithm does {\em not} query component $i$
in round $r$ (and zero otherwise).  To see how this is possible, consider the following truncations of \eqref{eq:DetLipF}:
\begin{equation}
f_i^t(x) = \frac{1}{\sqrt{2}}\abs{b - \inner{x}{v_0}} + \frac{1}{2\sqrt{k}} \sum_{r=1}^{t-1} \delta_{i,r} \abs{\inner{x}{v_{r-1}} - \inner{x}{v_{r}}}
\end{equation}

During each round $t$, the adversary answers queries  according to
$f_i^t$, which depends only on $v_r,\delta_{i,r}$ for $r<t$, i.e.~from
previous rounds.  When the round is completed, $\delta_{i,t}$ is
determined and $v_{t}$ is chosen to be orthogonal to the
vectors $\set{v_0,...,v_{t-1}}$ as well as every point queried by the
algorithm so far, thus defining $f_i^{t+1}$ for the next round.  In 
Appendix \ref{appendix:DetLBLower} we prove that these responses based on $f_i^t$ are
consistent with $f_i$.

The algorithm can only learn $v_{r}$ after it
completes round $r$---until then every iterate is orthogonal to it
by construction.  The average of these functions reaches its
minimum of $F(x^*)=0$ at $x^* = b\sum_{r=0}^{k}v_r$, so we can view
optimizing these functions as the task of discovering the vectors
$v_r$---even if only $v_k$ is missing, a suboptimality better than
$b/(6\sqrt{k})>\epsilon$ cannot be achieved.  Therefore, the
deterministic algorithm must complete at least $k$ rounds of
optimization, each comprising at least
$\left\lceil\frac{m}{2}\right\rceil$ queries to $h_F$ in order to
optimize $F$.  The key to this construction is that even though each term
$\abs{\inner{x}{v_{r-1}} - \inner{x}{v_{r}}}$ appears in $m/2$
components, and hence has a strong effect on the average $F(x)$, we
can force a deterministic algorithm to make $\Omega(m)$ queries during
each round before it finds the next relevant term.  We obtain (for
complete proof see Appendix \ref{appendix:DetLBLower}):

\begin{restatable}{theorem}{DetLBLower}\label{thm:DetLBLower}
For any $L,B > 0$, any $0 < \epsilon < \frac{LB}{12}$, any $m \geq 2$, and any deterministic algorithm $A$ with access to $h_F$, there exists a dimension $d = \mathcal{O}\left( \frac{mLB}{\epsilon} \right)$, and $m$ functions $f_i$ defined over $\mathcal{X} = \set{x \in \mathbb{R}^d: \norm{x}\leq B}$, which are convex and $L$-Lipschitz continuous, such that in order to find a point $\hat{x}$ for which $F(\hat{x}) - F(x^*) < \epsilon$, $A$ must make $\Omega\left( \frac{mLB}{\epsilon} \right)$ queries to $h_F$.
\end{restatable}

Furthermore, we can always reduce optimizing a function over $\norm{x}\leq B$
to optimizing a strongly convex function by adding the regularizer
{\small $\epsilon \norm{x}^2/(2 B^2)$} to each component, implying (see complete proof in Appendix \ref{appendix:DetLscLower}):

\begin{restatable}{theorem}{DetLscLower}\label{thm:DetLscLower}
For any $L,\lambda > 0$, any $0 < \epsilon < \frac{L^2}{288\lambda}$, any $m \geq 2$, and any deterministic algorithm $A$ with access to $h_F$, there exists a dimension $d=\mathcal{O}\left( \frac{mL}{\sqrt{\lambda\epsilon}} \right)$, and $m$ functions $f_i$ defined over $\mathcal{X} \subseteq \mathbb{R}^d$, which are $L$-Lipschitz continuous and $\lambda$-strongly convex, such that in order to find a point $\hat{x}$ for which $F(\hat{x}) - F(x^*) < \epsilon$, $A$ must make $\Omega\left( \frac{mL}{\sqrt{\lambda\epsilon}} \right)$ queries to $h_F$.
\end{restatable}

\subsection{Smooth Components}
When the components $f_i$ are required to be smooth, the lower
bound construction is similar to \eqref{eq:DetLipF}, except it is based on squared differences instead of
absolute differences. We consider the functions:
\begin{equation} \label{eq:DetSmoothF}
f_i(x) = \frac{1}{8} \left( \delta_{i,1}\left( \inner{x}{v_0}^2 -
    2a\inner{x}{v_0} \right) +  \delta_{i,k}\inner{x}{v_{k}}^2
     + \sum_{r=1}^{k} \delta_{i,r} \left(
    \inner{x}{v_{r-1}} - \inner{x}{v_{r}} \right)^2 \right) 
\end{equation}
where $\delta_{i,r}$ and $v_r$ are as before.  Again, we can
answer queries at round $t$ based only on $\delta_{i,r},v_r$ for
$r<t$.  This construction yields the following lower bounds (full
details in Appendix \ref{appendix:DetSBLower}):
\begin{restatable}{theorem}{DetSBLower}\label{thm:DetSBLower}
For any $\gamma,B,\epsilon > 0$, any $m \geq 2$, and any deterministic algorithm $A$ with access to $h_F$, there exists a sufficiently large dimension $d = \mathcal{O}\big( m\sqrt{\gamma B^2/\epsilon} \big)$, and $m$ functions $f_i$ defined over $\mathcal{X} = \set{x \in \mathbb{R}^d:\norm{x}\leq B}$, which are convex and $\gamma$-smooth, such that in order to find a point $\hat{x} \in \mathbb{R}^d$ for which $F(\hat{x}) - F(x^*) < \epsilon$, $A$ must make $\Omega\big( m\sqrt{\gamma B^2/\epsilon} \big)$ queries to $h_F$.
\end{restatable}
In the strongly convex case, we use a very similar construction, adding the term $\lambda\norm{x}^2/2$, which gives the following bound (see Appendix \ref{appendix:DetSscLower}):
\begin{restatable}{theorem}{DetSscLower}\label{thm:DetSscLower}
For any $\gamma,\lambda > 0$ such that $\frac{\gamma}{\lambda} > 73$, any $\epsilon > 0$, any $\epsilon_0 > \frac{3\gamma\epsilon}{\lambda}$, any $m \geq 2$, and any deterministic algorithm $A$ with access to $h_F$, there exists a sufficiently large dimension $d = \mathcal{O}\left( m\sqrt{\frac{\gamma}{\lambda}}\log\left(\frac{\lambda\epsilon_0}{\gamma\epsilon}\right) \right)$, and $m$ functions $f_i$ defined over $\mathcal{X} \subseteq \mathbb{R}^d$, which are $\gamma$-smooth and $\lambda$-strongly convex and where $F(0) - F(x^*) = \epsilon_0$, such that in order to find a point $\hat{x}$ for which $F(\hat{x}) - F(x^*) < \epsilon$, $A$ must make 
$\Omega\left( m\sqrt{\frac{\gamma}{\lambda}}\log\left(\frac{\lambda\epsilon_0}{\gamma\epsilon}\right) \right)$ queries to $h_F$.
\end{restatable}

\removed{ The four lower bounds presented in this section match the upper bounds
presented in sections \ref{sec:SmoothUpper} and \ref{sec:LipUppers} up to constant factors, and together show a
complete picture of the oracle complexity of problem \eqref{eq:main}
 for deterministic algorithms with access to $h_F$. }


\section{Lower Bounds for Randomized Algorithms}

We now turn to randomized algorithms for 
\eqref{eq:main}.  In the deterministic constructions, we
relied on being able to set $v_r$ and $\delta_{i,r}$ based on the
predictable behavior of the algorithm.  This is impossible for randomized algorithms,
 we must choose the ``hard'' function before we know the
random choices the algorithm will make---so the function
must be ``hard" more generally than before.

Previously, we chose vectors $v_r$ orthogonal to all previous
queries made by the algorithm. For randomized algorithms this cannot be ensured.
However, if we choose orthonormal vectors $v_r$ randomly in a
high dimensional space, they will be \emph{nearly} orthogonal to queries
with high probability. Slightly modifying the absolute or squared 
difference from before makes near orthogonality
sufficient. This issue increases the required dimension but does
not otherwise affect the lower bounds.

More problematic is our inability to anticipate the order in which
the algorithm will query the components, precluding the use of
$\delta_{i,r}$.  In the deterministic setting, if a term revealing a new $v_r$ appeared
 in half of the components, we could ensure that the algorithm must make $m/2$
queries to find it. However, a randomized algorithm could find it in two
queries in expectation, which would eliminate the linear dependence on $m$ in
the lower bound!  Alternatively, if only one component included
the term, a randomized algorithm would indeed need $\Omega(m)$
queries to find it, but that term's effect on suboptimality of $F$ would
be scaled down by $m$, again eliminating the dependence on $m$.

To establish a $\Omega(\sqrt{m})$ lower bound for randomized
algorithms we must take a new approach.  We define
$\left\lfloor\frac{m}{2}\right\rfloor$ pairs of functions which
operate on $\left\lfloor\frac{m}{2}\right\rfloor$ orthogonal subspaces
of $\mathbb{R}^d$. Each pair of functions resembles the constructions
from the previous section, but since there are many of them, the
algorithm must solve $\Omega(m)$ separate optimization problems in
order to optimize $F$.

\subsection{Lipschitz Continuous Components}
First consider the non-smooth, non-strongly-convex setting and assume
for simplicity $m$ is even (otherwise we simply let the last function be
zero).  We define the helper function $\psi_c$, which replaces the
absolute value operation and makes our construction resistant to
small inner products between iterates and not-yet-discovered
components:
\begin{equation}\label{main:defpsi}
\psi_c(z) = \max\left(0, \abs{z} - c\right)
\end{equation}
Next, we define $m/2$ pairs of functions, indexed by $i=1..m/2$:
\begin{align}\label{main:lnsc}
  f_{i,1}(x) &= \frac{1}{\sqrt{2}}\abs{b - \inner{x}{v_{i,0}}} + \frac{1}{2\sqrt{k}} \sum_{r\text{ even}}^k \psi_c\left(\inner{x}{v_{i,r-1}} - \inner{x}{v_{i,r}}\right) 
  \end{align}\begin{align*}
  f_{i,2}(x) &= \frac{1}{2\sqrt{k}} \sum_{r\text{ odd}}^k
  \psi_c\left(\inner{x}{v_{i,r-1}} - \inner{x}{v_{i,r}}\right)
\end{align*}
where $\set{v_{i,r}}_{r=0..k,i=1..m/2}$ are random orthonormal
vectors and $k=\Theta(\tfrac{1}{\epsilon\sqrt{m}})$.  With $c$ sufficiently small
and the dimensionality sufficiently high, with high probability the
algorithm only learns the identity of new vectors $v_{i,r}$ by 
alternately querying $f_{i,1}$ and $f_{i,2}$; so
revealing all $k+1$ vectors requires at least $k+1$ total queries.
Until $v_{i,k}$ is revealed, an iterate is
$\Omega(\epsilon)$-suboptimal on $(f_{i,1}+f_{i,2})/2$. From here, we show that an
$\epsilon$-suboptimal solution to $F(x)$ can be found only after at least $k+1$ 
queries are made to at least $m/4$ pairs, for a
total of $\Omega(mk)$ queries. This time, since the optimum $x^*$ will
need to have inner product $b$ with $\Theta(mk)$ vectors $v_{i,r}$, we
need to have
$b=\Theta(\tfrac{1}{\sqrt{mk}})=\Theta(\sqrt{\epsilon/\sqrt{m}})$, and the total number of queries is
$\Omega(mk)=\Omega(\tfrac{\sqrt{m}}{\epsilon})$.  The $\Omega(m)$ term of the lower bound
 follows trivially since we require $\epsilon = \mathcal{O}(1/\sqrt{m})$, (proofs in Appendix \ref{appendix:RandLBLower}):



\begin{restatable}{theorem}{RandLBLower}\label{thm:RandLBLower}
For any $L,B > 0$, any $0 < \epsilon < \frac{LB}{10\sqrt{m}}$, any $m \geq 2$, and any randomized algorithm $A$ with access to $h_F$, there exists a dimension $d = \mathcal{O}\parens{\frac{L^3B^3}{\epsilon^3\sqrt{m}}\log\parens{\frac{L^2B^2m}{\epsilon^2}}}$, and $m$ functions $f_i$ defined over $\mathcal{X} = \set{x \in \mathbb{R}^d: \norm{x}\leq B}$, which are convex and $L$-Lipschitz continuous, such that to find a point $\hat{x}$ for which $\mathbb{E}\left[F(\hat{x}) - F(x^*)\right] < \epsilon$, $A$ must make $\Omega\big( m + \frac{\sqrt{m}LB}{\epsilon} \big)$ queries to $h_F$.
\end{restatable}
An added regularizer 
gives the result for strongly convex functions (see Appendix \ref{appendix:RandLscLower}):
\begin{restatable}{theorem}{RandLscLower}\label{thm:RandLscLower}
For any $L,\lambda > 0$, any $0 < \epsilon < \frac{L^2}{200\lambda m}$, any $m \geq 2$, and any randomized algorithm $A$ with access to $h_F$, there exists a dimension $d = \mathcal{O}\left( \frac{L^3}{\sqrt{\lambda^3\epsilon^3 m}}\log \frac{L^2m}{\lambda\epsilon} \right)$, and $m$ functions $f_i$ defined over $\mathcal{X} \subseteq \mathbb{R}^d$, which are $L$-Lipschitz continuous and $\lambda$-strongly convex, such that in order to find a point $\hat{x}$ for which $\mathbb{E}\left[F(\hat{x}) - F(x^*)\right] < \epsilon$, $A$ must make $\Omega\big( m +  \frac{\sqrt{m}L}{\sqrt{\lambda\epsilon}} \big)$ queries to $h_F$.
\end{restatable}

The large dimension required by these lower bounds is the cost of omitting the assumption that the algorithm's queries lie in the span of previous oracle responses. If we do assume that the queries lie in that span, the necessary dimension is only on the order of the number of oracle queries needed.

When $\epsilon = \Omega\left(LB/\sqrt{m}\right)$ in the non-strongly convex case or $\epsilon=\Omega\left( L^2/(\lambda m) \right)$ in the strongly convex case, the lower bounds for randomized algorithms presented above do not apply. Instead, we can obtain a lower bound based on an information theoretic argument. We first uniformly randomly choose a parameter $p$, which is either $(1/2 -2\epsilon)$ or $(1/2 + 2\epsilon)$.
Then for $i = 1,...,m$, in the non-strongly convex case we make $f_i(x) = x$ with probability $p$ and $f_i(x) = -x$ with probability $1-p$.
Optimizing $F(x)$ to within $\epsilon$ accuracy then implies recovering the bias of the Bernoulli random variable, which requires $\Omega(1/\epsilon^2)$ queries based on a standard information theoretic result \citep{Agarwal09,LeCam}. Setting $f_i(x) = \pm x +\frac{\lambda}{2}\norm{x}^2$ gives a $\Omega(1/(\lambda\epsilon))$ lower bound in the $\lambda$-strongly convex setting. This is formalized in Appendix \ref{appendix:bigm}.

\subsection{Smooth Components}
When the functions $f_i$ are smooth and not strongly convex, we define another helper function
$\phi_c$:
\begin{equation}\label{main:defphi}
\phi_c(z) = \begin{cases} 0 & \abs{z} \leq c \\ 2(\abs{z} - c)^2 & c <
  \abs{z} \leq 2c \\ z^2 - 2c^2 & \abs{z} > 2c \end{cases} 
\end{equation}
and the following pairs of functions for $i = 1,...,m/2$: 
\begin{align}\label{main:snsc}
f_{i,1}(x) &= \frac{1}{16}\bigg( \inner{x}{v_{i,0}}^2 - 2a\inner{x}{v_{i,0}} + \sum_{r\text{ even}}^k \phi_c\left(\inner{x}{v_{i,r-1}} - \inner{x}{v_{i,r}}\right)  \bigg) \\
f_{i,2}(x) &= \frac{1}{16}\bigg(  \phi_c\left(\inner{x}{v_{i,k}}\right) + \sum_{r\text{ odd}}^k \phi_c\left(\inner{x}{v_{i,r-1}} - \inner{x}{v_{i,r}}\right)\bigg)\nonumber
\end{align}
with $v_{i,r}$ as before.  The same arguments apply, after replacing
the absolute difference with squared difference. A separate argument is required in this case for the $\Omega(m)$ term in the bound, which 
we show using a construction involving $m$ simple linear functions (see Appendix \ref{appendix:RandSBLower}).

\begin{restatable}{theorem}{RandSBLower}\label{thm:RandSBLower}
For any $\gamma,B,\epsilon > 0$, any $m \geq 2$, and any randomized algorithm $A$ with access to $h_F$, there exists a sufficiently large dimension $d = \mathcal{O}\parens{\frac{\gamma^2B^4}{\epsilon^2}\log\parens{\frac{m\gamma B^2}{\epsilon}}}$ and $m$ functions $f_i$ defined over $\mathcal{X} = \set{x \in \mathbb{R}^d:\norm{x}\leq B}$, which are convex and $\gamma$-smooth, such that to find a point $\hat{x} \in \mathbb{R}^d$ for which $\mathbb{E}\left[F(\hat{x}) - F(x^*)\right] < \epsilon$, $A$ must make $\Omega\left( m + \sqrt{\frac{m\gamma B^2}{\epsilon}} \right)$ queries to $h_F$.
\end{restatable}
In the strongly convex case, we add the term $\lambda\norm{x}^2/2$ to $f_{i,1}$ and $f_{i,2}$ (see Appendix \ref{appendix:RandSscLower}) to obtain: 
\begin{restatable}{theorem}{RandSscLower}\label{thm:RandSscLower}
For any $m \geq 2$, any $\gamma,\lambda>0$ such that $\frac{\gamma}{\lambda} > 161m$, any $\epsilon > 0$, any $\epsilon_0 > 60\epsilon\sqrt{\frac{\gamma}{\lambda m}}$, and any randomized algorithm $A$, there exists a dimension $d = \mathcal{O}\parens{\frac{\gamma m}{\lambda \epsilon}\log^4\parens{\frac{\epsilon_0^2 \lambda m}{\gamma \epsilon^2}}\log\parens{\frac{\gamma m^2}{\lambda \epsilon}}}$, domain $\mathcal{X}\subseteq\mathbb{R}^d$, $x_0 \in \mathcal{X}$, and $m$ functions $f_i$ defined on $\mathcal{X}$ which are $\gamma$-smooth and $\lambda$-strongly convex, and such that $F(x_0)-F(x^*)=\epsilon_0$ and such that in order to find a point $\hat{x} \in \mathcal{X}$ such that $\mathbb{E}\left[F(\hat{x}) - F(x^*)\right] < \epsilon$, $A$ must make $\Omega\left(m + \sqrt{\frac{m\gamma}{\lambda}}\log\left(\frac{\epsilon_0}{\epsilon}\sqrt{\frac{m\lambda}{\gamma}}\right)\right)$ queries to $h_F$.
\end{restatable}

\paragraph{Remark:}
We consider \eqref{eq:main} as a constrained optimization problem, thus the minimizer of $F$ could be achieved on the boundary of $\mathcal{X}$, meaning that the gradient need not vanish. If we make the additional assumption that the minimizer of $F$ lies on the \emph{interior} of $\mathcal{X}$ (and is thus the unconstrained global minimum), Theorems \ref{thm:DetLBLower}-\ref{thm:RandSscLower} all still apply, with a slight modification to Theorems \ref{thm:DetSBLower} and \ref{thm:RandSBLower}. Since the gradient now needs to vanish on $\mathcal{X}$, $0$ is always $\mathcal{O}(\gamma B^2)$-suboptimal, and only values of $\epsilon$ in the range $0 < \epsilon < \frac{\gamma B^2}{128}$ and $0 < \epsilon < \frac{9\gamma B^2}{128}$ result in a non-trivial lower bound (see Remarks at the end of Appendices \ref{appendix:DetSBLower} and \ref{appendix:RandSBLower}).

\section{Conclusion}
We provide a tight (up to a log factor) understanding of optimizing
finite sum problems of the form \eqref{eq:main} using a component prox
oracle.  

Randomized optimization of \eqref{eq:main} has been the subject of
much research in the past several years, starting with the
presentation of SDCA and SAG, and continuing with accelerated variants.  
Obtaining lower bounds can be very
useful for better understanding the problem, for knowing where it
might or might not be possible to improve or where different
assumptions would be needed to improve, and for establishing
optimality of optimization methods.
Indeed, several
attempts have been made at lower bounds for the finite sum setting
\cite{AgarwalBottou,Lan}.  But as we explain in the introduction, these
were unsatisfactory and covered only limited classes of methods.
Here we show that in a fairly general sense, accelerated SDCA, SVRG, SAG, 
and \textsc{Katyusha} are optimal up to a log factor.  Improving on their runtime
would require additional assumptions, or perhaps a stronger oracle. However, even 
if given ``full" access to the component functions, all algorithms that we can think
of utilize this information to calculate a prox vector. Thus, it is unclear what
realistic oracle would be more powerful than the prox oracle we consider.

Our results highlight the power of randomization, showing that no
deterministic algorithm can beat the linear dependence on $m$ and
reduce it to the $\sqrt{m}$ dependence of the randomized algorithms.

The deterministic algorithm for non-smooth problems that we present in
Section \ref{sec:LipUppers} is also of interest in its own right.  It
avoids randomization, which is not usually problematic, but makes it
fully parallelizable unlike the optimal stochastic methods.  Consider,
for example, a supervised learning problem where
$f_i(x)=\ell(\langle \phi_i,x\rangle,y_i)$ is the (non-smooth) loss on a single training example $(\phi_i,y_i)$, and the data is
distributed across machines.  Calculating a prox oracle involves
applying the Fenchel conjugate of the loss function $\ell$, but even if a
closed form is not available, this is often easy to compute numerically, and
is used in algorithms such as SDCA.  But unlike SDCA, which is
inherently sequential, we can calculate all $m$ prox operations in parallel on the
different machines, average the resulting gradients of the smoothed
function, and take an accelerated gradient step to implement our optimal deterministic
algorithm.  This method attains a
recent lower bound for distributed optimization, resolving a question
raised by \citet{Ohad}, and when the number of machines is
very large improves over all other known distributed optimization
methods for the problem.

In studying finite sum problems, we were forced to explicitly study
lower bounds for randomized optimization as opposed to stochastic optimization 
(where the source of randomness is the oracle, not the algorithm).  Even for the classic
problem of minimizing a smooth function using a first order oracle, we
could not locate a published proof that applies to randomized
algorithms.  We provide a simple construction using
$\epsilon$-insensitive differences that allows us to easily obtain such lower
bounds without reverting to assuming the iterates are spanned by
previous responses (as was done, e.g., in \cite{Lan}), and could
potentially be useful for establishing randomized lower bounds also in
other settings.

\paragraph{Acknowledgements:}
We thank Ohad Shamir for his helpful discussions and for pointing out \citep{HazanReduction}.

\bibliographystyle{plainnat}
{\footnotesize \bibliography{submission}}

\appendix
\section{Upper bounds for non-smooth sums \label{appendix:LipUpper}}
Consider the case where the components are not strongly convex. As shown in lemma \ref{lem:prox}, we can use a single call to a prox oracle to obtain the gradient of 
\begin{equation*}
f^{(\beta)}(x) = \inf_{u\in\mathcal{X}} f(u) + \frac{\beta}{2}\norm{x-u}^2
\end{equation*}
which is a $\beta$-smooth approximation to $f$. We then consider the new optimization problem:
\begin{equation} \label{eq:appendixFtilde}
\min_{x \in \mathcal{X}}\set{ \tilde{F}^{(\beta)}(x) := \frac{1}{m} \sum_{i=1}^m f^{(\beta)}_i(x) }.
\end{equation}
Also by lemma \ref{lem:prox}, setting $\beta = \frac{L^2}{\epsilon}$ ensures that $\tilde{F}^{(\beta)}(x) \leq F(x) \leq \tilde{F}^{(\beta)}(x) + \frac{\epsilon}{2}$ for all $x$. Consequently, any point which is $\frac{\epsilon}{2}$-suboptimal for $\tilde{F}^{(\beta)}$ will be $\epsilon$-suboptimal for $F$. This technique therefore reduces the task of optimizing an instance of an $L$-Lipschitz finite sum to that of optimizing an $\frac{L^2}{\epsilon}$-smooth finite sum.
 
Solving \eqref{eq:appendixFtilde} to $\frac{\epsilon}{2}$-suboptimality using AGD requires $\mathcal{O}\left( \frac{mLB}{\epsilon} \right)$ gradients for $\tilde{F}^{(\beta)}$ which requires that same number of prox oracles from $h_F$. Formally:

\begin{theorem}
For any $L,B> 0$, any $\epsilon < LB$, and any $m \geq 1$ functions $f_i$ which are convex and $L$-Lipschitz continuous over the domain $\mathcal{X} \subseteq \set{x\in\mathbb{R}^d:\norm{x} \leq B}$, applying AGD to \eqref{eq:appendixFtilde} for $\beta = \frac{L^2}{\epsilon}$, will result in a point $\hat{x}$ such that $F(\hat{x}) - F(x^*) < \epsilon$ after $\mathcal{O}\left( \frac{mLB}{\epsilon}  \right)$ queries to $h_F$.
\end{theorem}

When the component functions are $\lambda$-strongly convex, a more sophisticated strategy is required to avoid an extra $\log$ factor. The solution is the AdaptSmooth algorithm \citep{HazanReduction}. 
This involves solving $\mathcal{O}(\log\frac{1}{\epsilon})$ smooth and strongly convex subproblems, where the $t^{\text{th}}$ subproblem is reducing the suboptimality of the $\beta_t$-smooth and $\lambda$-strongly convex function $F^{(\beta_t)}(x)$ by a factor of four, where $\beta_t = \frac{L^2}{\epsilon_0}2^t$ and where $\epsilon_0 \leq \frac{L^2}{\lambda}$ upper bounds the initial suboptimality. Using this method results in an $\epsilon$-suboptimal solution for $F$ after $\sum_{t=0}^{\log\frac{\epsilon_0}{\epsilon}} \text{Time}(\beta_t, \lambda)$ queries to $h_F$. 

In the case of AGD, Time$(\gamma,\lambda) = \mathcal{O}\left( m\sqrt{\frac{\gamma}{\lambda}} \right)$ and
  \[ \sum_{t=0}^{\log\frac{\epsilon_0}{\epsilon}} \text{Time}\left(\frac{L^2}{\epsilon_0}2^t, \lambda\right) = \mathcal{O}\left( \frac{mL}{\sqrt{\lambda\epsilon}} \right) \]
 

\begin{theorem}
 For any $L,\lambda,\epsilon > 0$, and any $m \geq 1$ functions $f_i$, which are $L$-Lipschitz continuous and $\lambda$-strongly convex on the domain $\mathcal{X} \subseteq \mathbb{R}^d$, applying AdaptSmooth with AGD will find a point $\hat{x} \in \mathcal{X}$ such that $F(\hat{x}) - F(x^*) < \epsilon$ after $\mathcal{O}\left( \frac{mL}{\sqrt{\lambda\epsilon}} \right)$ queries to $h_F$.
\end{theorem}

To conclude our presentation of upper bounds, we emphasize that the smoothing methods described in this section will only improve oracle complexity when used with accelerated methods. For example, using non-accelerated gradient descent on $\tilde{F}^{(\beta)}$ in the not strongly convex case leads to an oracle complexity of $\mathcal{O}\left( \frac{mL^2B^2}{\epsilon^2} \right)$, which is no better than the convergence rate of gradient descent applied directly to $F$.

\section{Lower bounds for deterministic algorithms}
\subsection{Non-smooth and not strongly convex components \label{appendix:DetLBLower}}
\DetLBLower*
\begin{proof}
Without loss of generality, we can assume $L = B = 1$. For particular values $b$ and $k$ to be decided upon later, we use the functions \eqref{eq:DetLipF}:
\[ f_i(x) = \frac{1}{\sqrt{2}}\abs{b - \inner{x}{v_0}} + \frac{1}{2\sqrt{k}} \sum_{r=1}^{k} \delta_{i,r} \abs{\inner{x}{v_{r-1}} - \inner{x}{v_{r}}} \]
It is straightforward to confirm that $f_i$ is both $1$-Lipschitz and convex (for orthonormal vectors $v_r$ and indicators $\delta_{i,r} \in \set{0,1}$). As explained in the main text, the orthonormal vectors $v_r \in \mathbb{R}^d$ and indicators $\delta_{i,r} \in\set{0,1}$ are chosen according to the behavior of the algorithm $A$. At the end of each round $t$, we set $\delta_{i,t} = 1$ iff the algorithm did \emph{not} query function $i$ during round $t$ (and zero otherwise), and we set $v_t$ to be orthogonal to the vectors $\set{v_0,...,v_{t-1}}$ as well as every query made by the algorithm so far. Orthogonalizing the vectors in this way is possible as long as the dimension is at least as large as the number of oracle queries $A$ has made so far plus $t$. We are allowed to construct $v_t$ and $\delta_{i,t}$ in this way as long as the algorithm's execution up until round $t$, and thus our choice of $v_t$ and $\delta_{i,t}$, depends only on $v_r$ and $\delta_{i,r}$ for $r < t$. We can enforce this condition by answering the queries during round $t$ according to
\[ f_i^t(x) = \frac{1}{\sqrt{2}}\abs{b - \inner{x}{v_0}} + \frac{1}{2\sqrt{k}} \sum_{r=1}^{t-1} \delta_{i,r} \abs{\inner{x}{v_{r-1}} - \inner{x}{v_{r}}} \]
For non-smooth functions, the subgradient oracle is not
 uniquely defined----many different subgradients might be a valid response.  
 However, in order to say that an algorithm successfully optimizes a function, 
 it must be able to do so no matter which subgradient is receives. Conversely, to
show a lower bound, it is sufficient to show that for \emph{some} valid
subgradient the algorithm fails.  And so, in constructing a ``hard''
instance to optimize we are actually constructing both a function and
a subgradient oracle for it, with specific subgradient responses. Therefore, 
answering the algorithm's queries during round $t$ according to $f_i^t$ is 
valid so long as the subgradient we return is a valid subgradient for $f_i$ (the 
converse need not be true) and the prox returned is exactly the prox of $f_i$.
For now, assume that this query-answering strategy is consistent (we will prove this
last).

Then if $d = \lceil \frac{m}{\epsilon} \rceil + k + 1$ and if $x$ is an iterate generated both before $A$ completes round $k$ and before it makes $\lceil \frac{m}{\epsilon} \rceil$ queries to $h_F$ (so that the dimension is large enough to orthogonalize each $v_t$ as described above), then $\inner{x}{v_k} = 0$ by construction. This allows us to bound the suboptimality of $F(x)$ (since $\lceil\frac{m}{2}\rceil$ functions are queried during each round, $\sum_{i=1}^m\delta_{i,r} = \lfloor\frac{m}{2}\rfloor$):
\begin{align*}
F(x) &= \frac{1}{m}\sum_{i=1}^m f_i(x) \\
&= \frac{1}{\sqrt{2}}\abs{b - \inner{x}{v_0}} + \frac{\lfloor \frac{m}{2} \rfloor}{2m\sqrt{k}} \sum_{r=1}^{k} \abs{\inner{x}{v_{r-1}} - \inner{x}{v_{r}}}
\end{align*}
$F$ is non-negative and  $F(x_b) = 0$ where $x_b = b\sum_{r=0}^{k} v_r$. Choosing $b = \frac{1}{\sqrt{k+1}}$ makes $\norm{x_b} = 1$ so that $x_b \in \mathcal{X}$. Therefore, $F$ achieves its minimum on $\mathcal{X}$ and
\begin{align*}
F(x) - F(x^*) &= \frac{1}{\sqrt{2}}\abs{b - \inner{x}{v_0}} + \frac{\lfloor \frac{m}{2} \rfloor}{2m\sqrt{k}} \sum_{r=1}^{k} \abs{\inner{x}{v_{r-1}} - \inner{x}{v_{r}}} - 0 \\
&\geq \frac{1}{\sqrt{2}}\abs{b - \inner{x}{v_0}} + \frac{1}{6\sqrt{k}} \abs{\inner{x}{v_{0}} - \inner{x}{v_{k}}} \\
&= \frac{1}{\sqrt{2}}\abs{b - \inner{x}{v_0}} + \frac{1}{6\sqrt{k}}\abs{\inner{x}{v_{0}}} \\
&\geq \min_{z \in \mathbb{R}} \frac{1}{\sqrt{2}}\abs{b - z} + \frac{1}{6\sqrt{k}} \abs{z} \\
&= \frac{b}{6\sqrt{k}} \\
&\geq \frac{1}{12k}
\end{align*}
Where the final inequality holds when $k \geq 1$. Setting $k = \lfloor \frac{1}{12\epsilon} \rfloor$ implies $F(x) - F(x^*) \geq \epsilon$. Therefore, $A$ must either query $h_F$ more than $\lceil\frac{m}{\epsilon}\rceil$ times or complete $k$ rounds to reach an $\epsilon$-suboptimal solution. Completing each round requires at least $\lceil\frac{m}{2}\rceil$ queries to $h_F$, so when $\epsilon \leq \frac{1}{12}$, this implies a lower bound of 
\[ \min\left(\frac{m}{\epsilon},\  \left\lfloor \frac{1}{12\epsilon} \right\rfloor \frac{m}{2} \right) \geq \frac{m}{48\epsilon} \]

To complete the proof, it remains to show that the subgradients and proxs of $f_i^t$ are consistent with those of $f_i$ at every time $t$. Since every function operates on the $(k+1)$-dimensional subspace of $\mathbb{R}^d$ spanned by $\set{v_r}$, it will be convenient to decompose vectors into two components: $x = x^v + x^\perp$ where $x^v = \sum_{r=0}^{k}\inner{x}{v_r}v_r$ and $x^\perp = x - x^v$. Note that $f_i^t(x) = f_i^t(x^v)$.

\begin{lemma}
For any $t \leq k$ and any $x$ such that $x^v \in \spn{v_0, v_1, ..., v_{t-1}}$, if function $i$ is queried during round $t$, then $\partial f_i^t(x) \subseteq \partial f_i(x)$.
\end{lemma}
\begin{proof}
All subgradients of $f_i$ have the form
\[  \frac{\text{sign}\big(b - \inner{x}{v_0}\big)}{\sqrt{2}}v_0 + \frac{1}{2\sqrt{k}} \sum_{r=1}^{k} \delta_{i,r} \text{sign}\big(\inner{x}{v_{r-1}} - \inner{x}{v_{r}}\big)(v_{r-1} - v_{r}) \]
where we define $\text{sign}(0) = 0$. Since function $i$ is queried during round $t$, $\delta_{i,t} = 0$, and since $\inner{x}{v_{r-1}} = 0 = \inner{x}{v_{r}}$ for all $r > t$, $\partial f_i(x)$ contains all subgradients of the form
\[ \frac{\text{sign}\big(b - \inner{x}{v_0}\big)}{\sqrt{2}}v_0 + \frac{1}{2\sqrt{k}} \sum_{r=1}^{t-1} \delta_{i,r} \text{sign}\big(\inner{x}{v_{r-1}} - \inner{x}{v_{r}}\big)(v_{r-1} - v_{r}) \]
which is exactly $\partial f_i^t(x)$.
\end{proof}

\begin{lemma} \label{lem:DetLBProxSpan}
For any $t \leq k$ and any $x$ such that $x^v \in \spn{v_0, v_1, ..., v_{t-1}}$, if function $i$ is queried during round $t$ then $\forall \beta > 0$, $\prox_{f_i^t}(x,\beta) = \prox_{f_i}(x,\beta)$.
\end{lemma}
\begin{proof}
Consider the definition of the prox oracle from equation \ref{eq:prox}
\begin{align*}
\prox_{f_i}(x,\beta) &= \argmin_{u} f_i(u) + \frac{\beta}{2}\norm{x - u}^2 \\
&= \argmin_{u^v, u^\perp} f_i(u^v) + \frac{\beta}{2}\norm{x^v + x^\perp - u^v - u^\perp}^2 \\
&= \argmin_{u^v} f_i(u^v) + \frac{\beta}{2}\norm{x^v - u^v}^2 + \argmin_{u^\perp} \frac{\beta}{2}\norm{x^\perp - u^\perp}^2\\
&= x^\perp + \prox_{f_i}(x^v,\beta)
\end{align*}
Next, we further decompose $x^v = x^- + x^+$ where 
\[ x^- = \sum_{r=0}^{t-1} \inner{x^v}{v_r}v_r\qquad\text{ and }\qquad x^+ = \sum_{r=t}^{k} \inner{x^v}{v_r}v_r \]
Note that $x^+ = 0$ and since function $i$ is queried during round $t$, $\delta_{i,t} = 0$. Therefore,
\begin{align}
\prox_{f_i}(x^v,\beta) &= \argmin_{u^-, u^+} f_i(u^- + u^+) + \frac{\beta}{2}\norm{x^- - u^- -u^+}^2 \label{eq:DetLBProx}\\
&= \argmin_{u^-, u^+} \frac{1}{\sqrt{2}}\abs{b - \inner{u^-}{v_0}} + \frac{1}{2\sqrt{k}} \sum_{r=1}^{t-1} \delta_{i,r} \abs{\inner{u^-}{v_{r-1}} - \inner{u^-}{v_{r}}} \nonumber\\
&\qquad+ \frac{1}{2\sqrt{k}}\sum_{r=t+1}^{k} \delta_{i,r} \abs{\inner{u^+}{v_{r-1}} - \inner{u^+}{v_{r}}} + \frac{\beta}{2}\left(\norm{x^- - u^-}^2 + \norm{u^+}^2\right) \nonumber\\
&= \prox_{f_i^t}(x^v,\beta) \nonumber
\end{align} 
The last equality follows from the fact that that the minimization is completely separable between $u^-$ and $u^+$, allowing us to minimized over each variable separately. The terms containing $u^+$ are non-negative and can be simultaneously equal to $0$ when $u^+ = 0$. Therefore, $\prox_{f_i^t}(x,\beta) = \prox_{f_i}(x,\beta)$.
\end{proof}
These lemmas show that the subgradients and proxs of $f_i^t$ at vectors which are queried during round $t$ are consistent with the subgradients and proxs of $f_i$. This confirms that our construction is sound, and completes the proof.
\end{proof}

\subsection{Non-smooth and strongly convex components \label{appendix:DetLscLower}}
\DetLscLower*
\begin{proof}
Suppose towards contradiction that the contrary were true, and there is an $A$ which can find a point $\hat{x}$ for which $F(\hat{x}) - F(x^*) < \epsilon$ after at most $o\left( \frac{mL}{\sqrt{\lambda\epsilon}} \right)$ queries to $h_F$. Then $A$ could be used to minimize the sum $\tilde{F}$ of $m$ functions $\tilde{f}_i$, which are convex and $L$-Lipschitz continuous over a domain of $\set{x:\norm{x}\leq B}$ by adding a regularizer. Let
\[ F(x) = \frac{1}{m} \sum_{i=1}^m f_i(x) := \frac{1}{m} \sum_{i=1}^m \tilde{f}_i(x) + \frac{\lambda}{2}\norm{x}^2  \]
Note that $f_i$ is $\lambda$-strongly convex and since $\tilde{f}_i$ is $L$-Lipschitz on the $B$-bounded domain, $f_i$ is $(L + \lambda B)$-Lipschitz continuous on the same domain. Furthermore, by setting $\lambda = \frac{\epsilon}{B^2}$, 
\[ \tilde{F}(x) \leq F(x) \leq \tilde{F}(x) + \frac{\epsilon}{2B^2}\norm{x}^2 \leq \tilde{F}(x) + \frac{\epsilon}{2} \]
By assumption, $A$ can find an $\hat{x}$ such that $F(\hat{x}) - F(x^*) < \frac{\epsilon}{2}$ using $o\left( \frac{m(L + \lambda B)}{\sqrt{\lambda\epsilon}} \right) = o\left( \frac{mLB}{\epsilon} \right)$ queries to $h_F$, and
\[ \frac{\epsilon}{2} > F(\hat{x}) - F(x^*) \geq \tilde{F}(\hat{x}) - \tilde{F}(\tilde{x}^*) - \frac{\epsilon}{2}  \]
Thus $\hat{x}$ is $\epsilon$-suboptimal for $\tilde{F}$. However, this contradicts the conclusion of theorem \ref{thm:DetLBLower} when the parameters of the strongly convex problem correspond to parameters of a non-strongly convex problem to which theorem \ref{thm:DetLBLower} applies. In particular, for any values $L > 0$, $\lambda > 0$, $0 < \epsilon < \frac{L^2}{288\lambda}$, and dimension $d = \mathcal{O}\left( \frac{mL}{\sqrt{\lambda\epsilon}} \right)$ there is a contradiction.
\end{proof}

\subsection{Smooth and not strongly convex components \label{appendix:DetSBLower}}
\DetSBLower*
\begin{proof}
This proof will be very similar to the proof of theorem \ref{thm:DetLBLower}. Without loss of generality, we can assume that $\gamma = B  = 1$. For a values $a$ and $k$ to be fixed later, we define:
\[ f_i(x) = \frac{1}{8} \left( \delta_{i,1}\left( \inner{x}{v_0}^2 - 2a\inner{x}{v_0} \right)  + \sum_{r=1}^{k} \delta_{i,r} \left(\inner{x}{v_{r-1}} - \inner{x}{v_{r}} \right)^2 +  \delta_{i,k}\inner{x}{v_{k}}^2\right)  \]
We define the orthonormal vectors $v_r \in \mathbb{R}^d$ and indicators $\delta_{i,t} \in \set{0,1}$ as in the proof of theorem \ref{thm:DetLBLower}. That is, at the end of round $t$, we set $\delta_{i,t} = 1$ if the algorithm $A$ does \emph{not} query function $i$ during round $t$ (and zero otherwise) and we construct $v_t$ to be orthogonal to $\set{v_0,...,v_{t-1}}$ as well as every point queried by the algorithm so far. Orthogonalizing the vectors is possible as long as the dimension is at least as large as the number of oracle queries $A$ has made so far plus $t$. As before, we are allowed to construct $v_t$ and $\delta_{i,t}$ in this way as long as the algorithm's execution up until round $t$, and thus our choice of $v_t$ and $\delta_{i,t}$, depends only on $v_r$ and $\delta_{i,r}$ for $r < t$. We enforce this condition by answering the queries during round $t < k$ according to
\[ f_i^t(x) = \frac{1}{8} \left( \delta_{i,1}\left( \inner{x}{v_0}^2 - 2a\inner{x}{v_0} \right)  + \sum_{r=1}^{t-1} \delta_{i,r} \left(\inner{x}{v_{r-1}} - \inner{x}{v_{r}} \right)^2 \right) \]

We will assume for now that this query-answering strategy is self-consistent, and prove it later. This allows us to bound the suboptimality of $F(x)$. Note that since exactly $\lceil\frac{m}{2}\rceil$ functions are queried each round, $\sum_{i=1}^m\delta_{i,r} = \lfloor\frac{m}{2}\rfloor$, so let
\begin{align*}
F^t(x) &= \frac{1}{m}\sum_{i=1}^m f_i^t(x) + \delta_{i,t}\inner{x}{v_{t-1}}^2 \\
&= \frac{\lfloor\frac{m}{2}\rfloor}{8m} \left( \inner{x}{v_0}^2 - 2a\inner{x}{v_0}  + \sum_{r=1}^{t-1} \left(\inner{x}{v_{r-1}} - \inner{x}{v_{r}} \right)^2+ \inner{x}{v_{t-1}}^2 \right) 
\end{align*}
Then if $d = \lceil\frac{m}{\sqrt{\epsilon}}\rceil + k + 1$, and if $x$ is an iterate generated both before $A$ completes round $q := \lfloor\frac{k}{2}\rfloor$ and before it makes $\lceil\frac{m}{\sqrt{\epsilon}}\rceil$ queries to $h_F$, then $\inner{x}{v_r} = 0$ for all $r \geq q$ by construction. Then, for this $x$, $F^q(x) = F^{k+1}(x) = F(x)$. By first order optimality conditions for $F^t$, its optimum $x_t^*$ must satisfy that:
\begin{align*}
2\inner{x_t^*}{v_0} - \inner{x_t^*}{v_1} &= a \\
\inner{x_t^*}{v_{r-1}} - 2\inner{x_t^*}{v_{r}} + \inner{x_t^*}{v_{r+1}} &= 0 \quad\text{for } 1 \leq r \leq t-2\\
\inner{x_t^*}{v_{t-2}} - 2\inner{x_t^*}{v_{t-1}} &= 0
\end{align*}
It is straightforward to confirm that the solution to this system of equations is
\begin{equation*}
x_t^* = a\sum_{r=0}^{t-1} \left(1 - \frac{r + 1}{t+1}\right)v_r
\end{equation*}
that
\[ F^t(x_t^*) = -\frac{a^2\lfloor\frac{m}{2}\rfloor}{8m}\left( 1 - \frac{1}{t+1} \right) \]
and that
\begin{align*}
\norm{x_t^*}^2 &= a^2 \sum_{r=0}^{t-1} \left(1 - \frac{r+1}{t+1}\right)^2 \\
&= a^2 \left( t - \frac{2}{t+1}\sum_{r=0}^{t-1} (r+1) + \frac{1}{(t+1)^2}\sum_{r=0}^{t-1} (r+1)^2 \right)\\
&= a^2\left( t - \frac{2}{t+1}\frac{t(t+1)}{2} + \frac{1}{(t+1)^2}\frac{t(t+1)(2t+1)}{6} \right)\\
&\leq \frac{a^2t}{3}
\end{align*}
Thus, we set $a = \sqrt{\frac{3}{k+1}}$, ensuring $\norm{x_{k+1}^*} = 1$ so that $x_{k+1}^* = x^* \in \mathcal{X}$. Furthermore, for the iterate $x$ made before $q$ rounds of queries,
\begin{align*}
F(x) - F(x^*) &= F^q(x) - F^{k+1}(x_{k+1}^*) \\
&\geq F^q(x_q^*) - F^{k+1}(x_{k+1}^*) \\
&= -\frac{3\lfloor\frac{m}{2}\rfloor}{8m(k+1)}\left(1 - \frac{1}{\lfloor\frac{k}{2}\rfloor+1} \right) + \frac{3\lfloor\frac{m}{2}\rfloor}{8m(k+1)}\left(1 - \frac{1}{k+2} \right) \\
&\geq \frac{1}{32k^2} \\
\end{align*}
where the last inequality holds as long as $k \geq 2$. So, when $\epsilon < \frac{1}{128}$ and we let $k = \lfloor\frac{1}{\sqrt{32\epsilon}}\rfloor$, this ensures that 
\[ F(x) - F(x^*) = F^q(x) - F^{k+1}(x_{k+1}^*) \geq \epsilon \]
and therefore, $A$ must complete at least $q$ rounds or make more than $\lceil\frac{m}{\sqrt{\epsilon}}\rceil$ queries to $h_F$ in order to reach an $\epsilon$-suboptimal point. This implies a lower bound of 
\[ \min\left( \left\lceil\frac{m}{\sqrt{\epsilon}}\right\rceil,\ q\left\lceil\frac{m}{2}\right\rceil \right) \geq \frac{m}{16\sqrt{6\epsilon}} \]

To complete the proof, it remains to show that the gradient and prox of $f_i^t$ is consistent with those of $f_i$ at every time $t$. Since every function operates on the $(k+1)$-dimensional subspace of $\mathbb{R}^d$ spanned by $\set{v_r}$, it will be convenient to decompose vectors into two components: $x = x^v + x^\perp$ where $x^v = \sum_{r=0}^{k}\inner{x}{v_r}v_r$ and $x^\perp = x - x^v$. Note that $f_i^t(x) = f_i^t(x^v)$.

\begin{lemma}
For any $t \leq k$ and any $x$ such that $x^v \in \spn{v_0, v_1, ..., v_{t-1}}$, if function $i$ is queried during round $t$, then $\nabla f_i^t(x) =  \nabla f_i(x)$.
\end{lemma}
\begin{proof}
Since function $i$ is queried during round $t$, $\delta_{i,t} = 0$ so
\[ \nabla f_i(x) = \frac{1}{4}\left( \delta_{i,1}\left( \inner{x}{v_1}v_0 - 2v_0 \right) + \sum_{r=1}^{t-1} \delta_{i,r} \left( \inner{x}{v_{r-1}} - \inner{x}{v_{r}} \right)(v_{r-1} - v_{r}) \right) = \nabla f_i^t(x) \qedhere\]
\end{proof}

\begin{lemma}
For any $t \leq k$ and any $x$ such that $x^v \in \spn{v_0, v_1, ..., v_{t-1}}$, if function $i$ is queried during round $t$ then $\forall \beta > 0$, $\prox_{f_i^t}(x,\beta) = \prox_{f_i}(x,\beta)$.
\end{lemma}
\begin{proof}
Up until the last step, this proof is identical to the proof of lemma \ref{lem:DetLBProxSpan}, thus we pick up at \eqref{eq:DetLBProx}:
\begin{align}
\prox_{f_i}(x^v,\beta) &= \argmin_{u^-, u^+} f_i(u^- + u^+) + \frac{\beta}{2}\norm{x^- - u^- -u^+}^2 \nonumber\\
&= \argmin_{u^-, u^+} \frac{1}{8} \bigg(\delta_{i,1}\left( \inner{u^-}{v_0}^2 - 2a\inner{u^-}{v_0} \right)  + \sum_{r=1}^{t-1} \delta_{i,r} \left(\inner{u^-}{v_{r-1} - v_{r}} \right)^2 \nonumber\\
&+  \sum_{r=t+1}^{k} \delta_{i,r} \left(\inner{u^+}{v_{r-1} - v_{r}} \right)^2 + \delta_{i,k}\inner{u^+}{v_{k}}^2 \bigg) + \frac{\beta}{2}\left(\norm{x^- - u^-}^2 + \norm{u^+}^2\right) \nonumber\\
&= \prox_{f_i^t}(x^v,\beta) \nonumber
\end{align}
The final step comes from the fact that the $\argmin$ is separable over $u^-$ and $u^+$, meaning we can minimize the two terms individually. The terms which contain $u^+$ are non-negative and equal to zero when $u^+ = 0$.
\end{proof}
These lemmas show that the gradient and prox of $f_i^t$ at vectors which are queried during round $t$ are consistent with the gradient and prox of $f_i$. This confirms that our construction is sound.

This proves the lower bound for $\epsilon < \frac{\gamma B^2}{128}$, we can extend the same lower bound to $\epsilon \geq \frac{\gamma B^2}{128}$ using the following, very simple construction. Let
\[ f_i(x) = \begin{cases} 0 & \text{if function } i \text{ is queried in the first } m-1 \text{ queries} \\ 2m\epsilon\inner{x}{v} \end{cases} \]
where $v$ is a unit vector that is orthogonal to all of the first $m-1$ queries. This function is trivially $1$-smooth. By construction, the algorithm must make at least $m$ queries to learn the identity of $v$. Until it has done so, any iterate will have objective value zero, while the optimum $F(x^*) = F(v) = -2\epsilon$. Therefore, the algorithm must make at least $m$ queries to reach an $\epsilon$-suboptimal solution. For $\epsilon \geq \frac{\gamma B^2}{128}$
\[ m \geq m\sqrt{\frac{\gamma B^2}{128\epsilon}}  \]
Therefore a lower bound of $\Omega\left(m\sqrt{\frac{\gamma B^2}{\epsilon}} \right)$ applies for any $\epsilon > 0$.
\end{proof}

\paragraph{Remark:}
If make the additional assumption that $F$ is minimized on the interior of $\mathcal{X}$, since $0$ is $\mathcal{O}(\gamma B^2)$-suboptimal, only $0 < \epsilon < \frac{\gamma B^2}{128}$ gives a non-trivial lower bound. This lower bound is shown by the first construction presented in the previous proof.

\subsection{Smooth and strongly convex components \label{appendix:DetSscLower}}
\DetSscLower*
\begin{proof}
We will prove the theorem for $1$-smooth and $\lambda$-strongly convex components for any $\lambda < \frac{1}{73}$. This can be extended to arbitrary constants $\gamma$ and $\lambda'$ by taking $\lambda = \frac{\lambda'}{\gamma}$.

For any $k$ and for $\zeta$ and $C$ to be defined later, let
\begin{align*}
f_i(x) = \frac{1-\lambda}{8}\bigg( \delta_{i,1}\left(\inner{x}{v_0}^2 -2C\inner{x}{v_0}\right)& + \delta_{i,k}\zeta\inner{x}{v_k}^2 + \sum_{r=1}^k \delta_{i,r}\inner{x}{v_{r-1}-v_r}^2 \bigg) + \frac{\lambda}{2}\norm{x}^2
\end{align*}
where the vectors $v_r$ and indicators $\delta_{i,r}$ are defined in the same way as in the previous proof. This function is just a multiple of the construction in the proof of theorem \ref{thm:DetSBLower} plus the $\lambda\norm{x}^2/2$ term. It is clear that the norm term is uninformative for learning the identity of vectors $v_r$, as the component of the gradients and proxs which is due to that term is simply a scaling of the query point. Thus, for any iterate $x$ generated by $A$ before completing $t$ rounds of optimization, $\inner{x}{v_r} = 0$ for all $r \geq t$ so long as the dimension is greater than the total number of queries made to $h_F$ so far plus $k+1$; a fact which follows directly from the previous proof.

Since exactly $\lceil\frac{m}{2}\rceil$ functions are queried per round, $\sum_{i=1}^m\delta_{i,r} = \lfloor\frac{m}{2}\rfloor$ and thus
\begin{align*}
F(x) &= \frac{\lambda(Q-1)}{8}\left( \inner{x}{v_0}^2 -2C\inner{x}{v_0} + \zeta\inner{x}{v_k}^2 + \sum_{r=1}^k \inner{x}{v_{r-1}-v_r}^2 \right) + \frac{\lambda}{2}\norm{x}^2 
\end{align*}
where 
\[  Q = \frac{\lfloor \frac{m}{2} \rfloor}{m}(\frac{1}{\lambda} - 1) + 1  \]
By the first order optimality conditions for $F(x)$, its optimum $x^{*}$ must satisfy that:
\begin{align*}
2\frac{Q+1}{Q-1}\inner{x^{*}}{v_0} - \inner{x^{*}}{v_1} &= C \\
\inner{x^{*}}{v_{r-1}} - 2\frac{Q+1}{Q-1}\inner{x^{*}}{v_r} + \inner{x^{*}}{v_{r+1}} &= 0 \\
\left( 1 + \zeta + \frac{4}{Q-1} \right)\inner{x^{*}}{v_k} - \inner{x^{*}}{v_{k-1}} &= 0
\end{align*}
Defining $q := \frac{\sqrt{Q}-1}{\sqrt{Q}+1} < 1$ and setting $\zeta = 1-q$, it is straightforward to confirm that 
\[ x^{*} = C\sum_{r=0}^k q^{r+1}v_r \]
and also that 
\[ F(x^*) = -\frac{\lambda C^2}{8}\left( \sqrt{Q} - 1 \right)^2 \]

Thus, $F(0) - F(x^*) = \frac{\lambda C^2}{8}\left( \sqrt{Q} - 1 \right)^2$, so by choosing $C$ appropriately, we can make the initial suboptimality of our construction take any value $\epsilon_0$. Since $F$ is $\lambda$-strongly convex, for any $x_t$, $F(x_t) - F(x^*) \geq \frac{\lambda}{2}\norm{x_t - x^*}^2$. Let $x_t$ be an iterate which is generated before $t$ rounds of optimization have been completed, implying that $\inner{x_t}{v_r} = 0$ for all $r \geq t$. So
\begin{align*}
\frac{F(x_t)-F(x^*)}{F(0)-F(x^*)} &\geq \frac{\frac{\lambda}{2}\norm{x_t-x^*}^2}{\frac{\lambda C^2}{8}\left( \sqrt{Q} - 1 \right)^2} \\
&\geq \frac{4}{C^2}\frac{C^2\sum_{r=t}^k q^{2r+2} }{\left( \sqrt{Q} - 1 \right)^2} \\
&= \frac{4(q^{2t+2} - q^{2k+4})}{(1-q^2)\left( \sqrt{Q} - 1 \right)^2} \\
&= \frac{(q^{2t} - q^{2k+2})}{\sqrt{Q}} 
\end{align*}
If we set $k+1 = \left\lceil t - \frac{1}{2\log q} \right\rceil$ then
\begin{align*}
\frac{F(x_t)-F(x^*)}{F(0)-F(x^*)} &\geq \frac{(q^{2t} - q^{2k+2})}{\sqrt{Q}}  \\
&\geq \frac{q^{2t}}{2\sqrt{Q}} \\
&= \frac{1}{2\sqrt{Q}}\exp\left( -2t\log\frac{1}{q} \right) \\
&= \frac{1}{2\sqrt{Q}}\exp\left( -2t\log\left( 1+ \frac{2}{\sqrt{Q}-1}\right) \right) \\
&\geq \frac{1}{2\sqrt{Q}}\exp\left(\frac{-4t}{\sqrt{Q}-1} \right) \\
\end{align*}
and when $t = \left\lfloor \frac{\sqrt{Q}-1}{4}\log\frac{\epsilon_0}{2\sqrt{Q}\epsilon} \right\rfloor$
\begin{align*}
\frac{F(x_t)-F(x^*)}{F(0)-F(x^*)} &\geq \frac{\epsilon}{\epsilon_0}
\end{align*}
Therefore, the algorithm must complete at least $t$ rounds of queries before it can reach an $\epsilon$-suboptimal point. 
If $\epsilon_0 > \frac{3\epsilon}{\lambda}$ and $\lambda < \frac{1}{73}$,
since each round includes at least $\frac{m}{2}$ oracle queries, this implies a lower bound of
\[  \frac{m}{2} \left\lfloor \frac{\sqrt{Q}-1}{4}\log\frac{\epsilon_0}{2\sqrt{Q}\epsilon} \right\rfloor \geq \frac{m}{40\sqrt{\lambda}}\log\frac{\epsilon_0}{2\sqrt{Q}\epsilon} \geq \frac{m}{40\sqrt{\lambda}}\log\frac{\sqrt{\lambda}\epsilon_0}{2\epsilon} \]
queries to $h_F$ in order to reach an $\epsilon$-suboptimal point.
\end{proof}

\section{Lower bounds for randomized algorithms}
We thank Yair Carmon, John Duchi, Oliver Hinder, and Aaron Sidford for pointing out an issue with the original proofs in this section. We have since modified this section in a manner resembling \cite{Carmon17,woodworth2017lower}.

To prove the deterministic lower bounds, we constructed vectors $v_r$ adversarially,
orthogonalizing them to queries made by the algorithm. In the randomized setting, this
is impossible, as we cannot anticipate query points. Our solution was to instead draw
a set of ``important directions'' $v_{i,r}$ randomly in high dimensions. The intuition is that
a given vector, in this case the query made by the algorithm, will have a very small
inner product with a random unit vector with high probability if the dimension is large enough. 

Using this fact, we construct helper functions $\psi_c$ and $\phi_c$ to replace the 
absolute and squared difference functions used in the deterministic lower bounds. These
functions are both flat at 0 on the interval $[-c, c]$, meaning that the algorithm's query
needs to have a significant inner product with $v_{i,r}$ before the oracle needs to give 
that vector away as a gradient or prox. We will show that each of our constructions satisfy the following property:
\begin{restatable}{property}{keyprop}\label{keyprop}
For all $i$, all $t \leq k$ and $x$ such that $\forall r \geq t$ $\abs{\inner{x}{v_{i,r}}} \leq \frac{c}{2}$,
$\exists g_1 \in \partial f_{i,1}(x)$ and $\exists g_2 \in \partial f_{i,2}(x)$ such that
\[ g_1, g_2, \prox_{f_{i,1}}(x,\beta), \prox_{f_{i,2}}(x,\beta) \in \spn{x,v_{i,0},...,v_{i,t}}. \]
and such that $g_1$, $g_2$, $\prox_{f_{i,1}}(x,\beta)$, and $\prox_{f_{i,2}}(x,\beta)$ are all a deterministic function of $x^{(t)}$ and $\set{v_{i,r} : r < t}$.
\end{restatable}
In other words, when $x$ has a small inner product with $v_{i,r}$ for all $r \geq t$, then querying either $f_{i,1}$ or $f_{i,2}$ at $x$ will reveal at most $v_{i,t}$. Our bounds on the complexity of optimizing our functions are based on the principle that the algorithm can only learn one $v_{i,r}$ per query, so we need to control the probability that the premises of this property hold for \emph{every} query made by the algorithm. In this section, we will bound how large the dimensionality of the problem needs to be to ensure that with high probability, only one vector is revealed to the algorithm by each oracle response. 

For $i \in \set{1,2,\dots,m/2}$, we will define pairs of component functions $f_{i,1}$ and $f_{i,2}$ that satisfy Property \ref{keyprop}. If $m$ is odd, we set $f_m \equiv 0$ and lose only a factor of $(m-1)/m$ in our lower bound, therefore, we proceed assuming $m$ is even.

The function instances $f_{i,1}$ and $f_{i,2}$ are defined in terms of random vectors $\set{v_{i,r}}_{r=1}^k$. Specifically, let $\set{v_{i,r}}_{i\in[m/2], r \in [k]}$ be \emph{orthonormal} vectors drawn uniformly at random from the set of orthonormal vectors in $\mathbb{R}^d$ where $d \geq mk/2$.

The proofs will proceed by showing that optimizing $F(x) = \frac{1}{m}\sum_{i=1}^{m/2} f_{i,1}(x) + f_{i,2}(x)$ is equivalent to finding an $x$ with a significant inner product with a large number of the vectors $v_{i,r}$. This property depends on the specific constructions and will be discussed in detail later.

First, we will show generally that when the dimension is sufficiently large, Property \ref{keyprop} ensures each oracle access will ``reveal'' at most a single vector $v_{i,r}$ with high probability over the randomness in the draw of the vectors $v_{i,r}$. As a consequence of this fact, $\Omega(mk)$ accesses will be needed to optimize $F$.

We will prove the lower bound on the number of oracle accesses needed for an arbitrary \emph{deterministic} optimization algorithm to optimize our random finite sum instance, and then apply Yao's minimax principle in order to prove the lower bound for randomized optimization algorithms. 

We denote the $n^{th}$ query made by the algorithm $q^{(n)} = \left(i^{(n)}, j^{(n)}, x^{(n)}, \beta^{(n)}\right)$, which is a query to function $f_{i^{(n)}, j^{(n)}}$ at the point $x^{(n)}$ with the prox parameter $\beta^{(n)}$ (if applicable). For now, we require that $\exists B \geq 1$ s.t. $\norm{x^{(n)}} \leq B$ for all $n$; this assumption will be removed in the individual lower bound proofs. Recalling that we are considering for now deterministic optimization algorithms, the $n^{th}$ query is a function of the previous $n-1$ queries and the oracle's responses to those queries.

Let $\#(i,n) = \abs{\set{t < n : i^{(t)} = i}}$ denote the number of times that the functions $f_{i,1}$ or $f_{i,2}$ were queried before time $n$. Define $K_n = \set{v_{i,r} : i \in [m/2], r \leq \#(i,n)}$. These are the set of vectors $v_{i,r}$ that are supposed to be ``known'' to the optimization algorithm at time $n$. Also, let $U_n = \set{v_{i,r} : i \in [m/2], r > \#(i,n)}$. These are the set of vectors that are supposed to be ``unknown'' to the algorithm. 

Let $X_n = \set{x^{(1)},\dots,x^{(n)}}$ be the set of query points up to time $n$. 

Let $S_n = \spn{X_{n-1} \cup K_{n}}$, and let $S_n^\perp$ be its orthogonal complement. Let $P_n w$ be the orthogonal projection of a vector $w$ onto the subspace $S_n$, and let $P_n^\perp w$ be its orthogonal projection onto $S_n^\perp$.




We would like to show that the following event occurs with high probability over the draw of the vectors $v_{i,r}$:
\begin{equation}
E := \left[ \forall n\leq N\ \forall v \in U_n \ \abs{\inner{x^{(n)}}{v}} \leq \frac{c}{2} \right]
\end{equation}
Here, $c \leq 1/\sqrt{k}$ is a constant that will be determined later. This event indicate that the premises of Property \ref{keyprop} are satisfied, which will be used to prove the lower bounds in the next sections.

Define the following ``good'' event:
\begin{equation}
G_n = \left[ \forall v \in U_n\ \abs{\inner{\normalized{P^\perp_{n} x^{(n)}}}{ v}} < \alpha \right]
\end{equation}
where $\alpha = \frac{c}{2B(1 + \sqrt{2N})}$ for some $c \leq \frac{1}{\sqrt{N}}$. The events $G_n$ are useful because:
\begin{lemma} \label{lem:keyrand1}
$\bigcap_{n=1}^N G_n \implies E$.
\end{lemma}
\begin{proof}
This proof closely follows \cite[Lemma 9]{Carmon17} and \cite[Lemma 1]{woodworth2017lower}.

Let $G_{\leq n}$ denote $\bigcap_{t=1}^n G_t$. We will establish that for each $n \leq N$, $G_{\leq n} \implies \forall v \in U_n$ $\abs{\inner{x^{(t)}}{v}} \leq \frac{c}{2}$. 
To begin, we rewrite 
\begin{align}
\abs{\inner{x^{(n)}}{v}} 
&\leq B\abs{\inner{\frac{x^{(n)}}{\norm{x^{(n)}}}}{P_n v}} + B\abs{\inner{\frac{x^{(n)}}{\norm{x^{(n)}}}}{P_n^\perp v}} \\
&\leq B\norm{P_n v} + B\abs{\inner{\frac{P_n^\perp x^{(n)}}{\norm{x^{(n)}}}}{v}} \\
&\leq B\norm{P_n v} + B\abs{\inner{\frac{P_n^\perp x^{(n)}}{\norm{P_n^\perp x^{(n)}}}}{v}} \\
&\leq B\norm{P_n v} + B\alpha \label{eq:complicated-lemma-proof-original-equation}
\end{align}
First, we decomposed $v = P_n v + P_n^\perp v$ and used $\norm{x^{(t)}} \leq B$. Next, we used the Cauchy-Schwarz inequality and the self-adjointness of $P_n^\perp$. Finally, we used the non-expansivity of $P_n^\perp$ and applied the definition of $G_n$.

In order to bound the first term in \eqref{eq:complicated-lemma-proof-original-equation}, we will prove by induction that for $n \leq N$ and $v \in U_n$, $G_{<n} \implies \norm{P_n v}^2 \leq 2(n-1)\alpha^2$. The base case $n=1$ is trivial as $P_1 v$ is the projection of $v$ onto $S_1 = \varnothing$. 

For the inductive step, fix $n \leq N$ and $v \in U_n$, and assume the statement holds for $t < n$. Let $\hat{P}_n$ project onto $\spn{X_{n} \cup K_n}$ (this includes $x^{(n)}$ in contrast with $P_n$), and let $\hat{P}_n^\perp$ project onto the orthogonal subspace. 

Denote $v^{(t)} = v_{i^{(t)}, \#(i^{(t)},t)+1} \in U_{t} \cap K_{t+1}$, which is the vector that ``becomes known'' after time $t$. By definition, the set 
\begin{equation}
X_{n-1} \cup K_n = \set{x^{(1)},\dots, x^{(n-1)}} \cup \set{v^{(1)},\dots,v^{(n-1)}}
\end{equation}
spans $S_n$. Therefore, the Gram-Schmidt vectors
\begin{equation}
    \set{\normalized{P_t^\perp x^{(t)}}: t < n}
    \cup
    \set{\normalized{\hat{P}_t^\perp v^{(t)}}: t < n}
\end{equation}
form an orthonormal basis for $S_n$ (after ignoring any zero vectors that may arise from the projection).
We now write $\norm{P_n v}^2$ in terms of this orthonormal basis:
\begin{align}
\norm{P_n v}^2 
&= 
\sum_{t < n}\inner{\normalized{P_t^\perp x^{(t)}}}{v}^2
+
\sum_{t < n} \inner{\normalized{\hat{P}_t^\perp v^{(t)}}}{v}^2 \\
&\leq (n-1)\alpha^2 
+ \sum_{t < n}\frac{1}{\norm{\hat{P}_t^\perp v^{(t)}}^2} \inner{\hat{P}_t^\perp v^{(t)}}{v}^2\label{eq:second-term-bound-rand-algs-lemma}
\end{align}
We now bound the second term of \eqref{eq:second-term-bound-rand-algs-lemma}. Note that $v^{(t)} \not\in U_n$ is orthogonal to $v \in U_n$. Consider
\begin{align}
\abs{\inner{\hat{P}_t^\perp v^{(t)}}{v}}
&= \abs{\inner{v^{(t)}}{v} - \inner{\hat{P}_t v^{(t)}}{v}} \\
&= \abs{0 - \inner{P_t v^{(t)}}{v} - \inner{\inner{\normalized{P_t^\perp x^{(t)}}}{v^{(t)}}\normalized{P_t^\perp x^{(t)}}}{v}} \\
&\leq \abs{\inner{P_t v^{(t)}}{P_t v}} + \abs{\inner{\normalized{P_t^\perp x^{(t)}}}{v^{(t)}}}\abs{\inner{\normalized{P_t^\perp x^{(t)}}}{v}} \\
&\leq \norm{P_t v^{(t)}}\norm{P_t v} + \abs{\inner{\normalized{P_t^\perp x^{(t)}}}{v^{(t)}}}\abs{\inner{\normalized{P_t^\perp x^{(t)}}}{v}} 
\end{align}
Here, we used the triangle and Cauchy-Schwarz inequalities. Since $v^{(t)},v \in U_n$, by the inductive hypothesis
\begin{equation}
    \norm{P_t v^{(t)}}\norm{P_t v} 
    \leq 
    \sqrt{2(t-1)\alpha^2}\cdot\sqrt{2(t-1)\alpha^2}
    =
    2(t-1)\alpha^2
\end{equation}
and $G_{t}$ ensures
\begin{equation}
    \abs{\inner{\normalized{P_t^\perp x^{(t)}}}{v^{(t)}}}\abs{\inner{\normalized{P_t^\perp x^{(t)}}}{v}}
    \leq
    \alpha^2
\end{equation}
Thus, we conclude
\begin{equation}\label{eq:keylemma1-first-part}
\abs{\inner{\hat{P}_t^\perp v^{(t)}}{v}}
\leq (2t - 1)\alpha^2
\end{equation}

Furthermore, looking at the denominator of \eqref{eq:second-term-bound-rand-algs-lemma}
\begin{align}
\norm{\hat{P}_t^\perp v^{(t)}}^2 
&= \inner{\hat{P}_t^\perp v^{(t)}}{\hat{P}_t^\perp v^{(t)}} \\
&= \inner{\hat{P}_t^\perp v^{(t)}}{v^{(t)}} \\
&= \inner{v^{(t)}}{v^{(t)}} - \inner{\hat{P}_t v^{(t)}}{v^{(t)}} \\
&= 1 - \norm{P_t v^{(t)}}^2 - \inner{\normalized{P_t^\perp x^{(t)}}}{v^{(t)}}^2 \\
&\geq 1 - 2(t-1)\alpha^2 - \alpha^2
\end{align}
The inequality uses the inductive hypothesis and $G_t \impliedby G_{< n}$. Since $c \leq 1/\sqrt{N}$ and $B \geq 1$, $\alpha = \frac{c}{2B\parens{1 + \sqrt{2N}}} \leq \frac{1}{\sqrt{4N-2}}$, so this quantity is at least $1/2$.

Combining this and \eqref{eq:keylemma1-first-part} with \eqref{eq:second-term-bound-rand-algs-lemma}, we conclude
\begin{align}
\norm{P_n v}^2 
&\leq (n-1)\alpha^2 
+ \sum_{t < n}\frac{1}{\norm{\hat{P}_t^\perp v^{(t)}}^2} \inner{\hat{P}_t^\perp v^{(t)}}{v}^2 \\
&\leq (n-1)\alpha^2 
+ \sum_{t < n}2 (2t - 1)^2\alpha^4 \\
&= (n-1)\alpha^2 + \frac{2}{3}(n-1)(4(n-1)^2 - 1)\alpha^4 \\
&\leq (n-1)\alpha^2\parens{1 + 4n^2\alpha^2} \\
&= (n-1)\alpha^2\parens{1 + \frac{4n^2c^2}{4B^2\parens{1 + \sqrt{2N}}^2}} \\
&\leq (n-1)\alpha^2\parens{1 + \frac{N^2c^2}{1+2N}} \\
&\leq (n-1)\alpha^2\parens{1 + Nc^2} \\
&\leq 2(n-1)\alpha^2
\end{align}
We conclude that for all $n \leq N$ and $v \in U_n$, $G_{<n} \implies \norm{P_n v}^2 \leq 2(n-1)\alpha^2$.

Finally, we return to \eqref{eq:complicated-lemma-proof-original-equation} and conclude
\begin{align}
\abs{\inner{x^{(n)}}{v}} 
&\leq B\norm{P_n v} + B\alpha \\
&\leq B\sqrt{2(n-1)\alpha^2} + B\alpha \\
&< B\alpha\parens{\sqrt{2N} + 1} \\
&= \frac{Bc}{2B\parens{1 + \sqrt{2N}}}\parens{\sqrt{2N} + 1} \\
&= \frac{c}{2}
\end{align}
Since this holds for any $n \leq N$ and $v \in U_n$, we conclude $G_{\leq N} \implies E$.
\end{proof}

With Lemma \ref{lem:keyrand1} having been proven, we now observe the following corollary:
\begin{corollary}\label{cor:keyrand2}
For any $n \leq N$, let $o^{(n)}$ be the output of the gradient or prox oracle at time $n$. Then $G_{\leq n} \implies o^{(n)} \in \spn{\set{x^{(n)}} \cup K_{n+1}}$. 
\end{corollary}
\begin{proof}
By Lemma \ref{lem:keyrand1}, $G_{\leq n}$ implies $\forall r > \#(i^{(n)},n)$ $\abs{\inner{x^{(n)}}{v_{i^{(n)},r}}} \leq \frac{c}{2}$. Therefore, Property \ref{keyprop} immediately completes the proof.
\end{proof}

By Lemma \ref{lem:keyrand1} and the chain rule of probability, 
\begin{equation}
    \mathbb{P}[E] \geq \mathbb{P}\left[ G_{\leq N} \right] = \prod_{n=1}^N \mathbb{P}\left[ G_{n} \middle| G_{<n} \right]
\end{equation}
We address a single term in the product using the following lemma:
\begin{lemma}\label{lem:keyrand3}
For any $n \leq N$
\[
\mathbb{P}\left[ G_{n} \middle| G_{<n} \right]
\geq
1 - \frac{mk}{2}\exp\parens{-\frac{\alpha^2\parens{d - 2n + 2}}{2}}
\]
\end{lemma}
\begin{proof}
This proof closely follows \cite[Lemma 11]{Carmon17} and \cite[Lemma 3]{woodworth2017lower}


Fix $n \leq N$. We will start by proving the density $p_{U_n}(U_n | G_{<n}, K_n)$ is invariant to a rotations that preserve $S_n$. To that end, let $R$ be any rotation such that for all $w \in S_n$ $Rw = w$.  Then
\begin{gather}
p_{U_n}(U_n | G_{<n}, K_n) 
= \frac{\mathbb{P}\parens{G_{<n} \middle| U_n, K_n}p_{U_n, K_n}(U_n, K_n)}{\mathbb{P}\parens{G_{<n} \middle| K_n}p_{K_n}(K_n)}
\\
p_{U_n}(RU_n | G_{<n}, K_n) 
= \frac{\mathbb{P}\parens{G_{<n} \middle| RU_n, K_n}p_{U_n, K_n}(RU_n, K_n)}{\mathbb{P}\parens{G_{<n} \middle| K_n}p_{K_n}(K_n)}
\end{gather}

These expressions have the same denominator, so we will prove that $\mathbb{P}\parens{G_{<n} \middle| RU_n, K_n} = \mathbb{P}\parens{G_{<n} \middle| U_n, K_n}$ and $p_{U_n, K_n}(RU_n, K_n) = p_{U_n, K_n}(U_n, K_n)$. 

Beginning with the latter, $U_n \cup K_n$ is distributed uniformly on the set of orthonormal vectors in $\mathbb{R}^d$, therefore, as long as $RU_n \cup K_n$ is an orthonormal set, the densities are equal. The set $K_n$ is orthonormal by definition. The set $RK_n$ is orthonormal because for any $u_1,u_2 \in U_n$
\begin{gather}
\inner{Ru_1}{Ru_1} = \inner{R^\top R u_1}{u_1} = \inner{u_1}{u_1} = 1 \\
\inner{Ru_1}{Ru_2} = \inner{R^\top R u_1}{u_2} = \inner{u_1}{u_2} = 0
\end{gather}
Finally, for any $u \in U_n$, $k \in K_n$
\begin{equation}
\inner{Ru}{k} = \inner{Ru}{Rk} = \inner{R^\top R u}{k} = \inner{u}{k} = 0
\end{equation}
Therefore, $p_{U_n, K_n}(RU_n, K_n) = p_{U_n, K_n}(U_n, K_n)$.

Recall that we need only consider \emph{deterministic} optimization algorithms for the moment. Therefore, the algorithm's sequence of iterates are a deterministic function of the problem instance, which is determined by $U_n \cup K_n$. Therefore, the event $G_{<n}$ is determined by $U_n, K_n$ and $\mathbb{P}\parens{G_{<n} \middle| U_n, K_n}, \mathbb{P}\parens{G_{<n} \middle| RU_n, K_n} \in \set{0,1}$.  Let $x^{(1)},\dots,x^{(n)}$ be the algorithm's sequence of iterates for the problem instance determined by $U_n \cup K_n$, and let $x'^{(1)},\dots,x'^{(n)}$ be the algorithm's sequence of iterates for the problem instance determined by $RU_n \cup K_n$. 

We will prove by induction that for all $t \leq n$, $\mathbb{P}\parens{G_{<t} \middle| U_t, K_t} = \mathbb{P}\parens{G_{<t} \middle| RU_t, K_t}$ and if these probabilities are one, then $x^{(t)} = x'^{(t)}$. 

As the base case, consider $t = 2$. Since $G_{<2} = G_1$ depends only on the vectors $U_2 \cup K_2$ and the first iterate $x^{(1)}$, which is determined before making any oracle accesses and is thus independent of the vectors $U_2 \cup K_2$. Consequently, 
\begin{equation}
    \mathbb{P}\parens{G_{<2} \middle| U_2, K_2} = \mathbb{P}\parens{G_{1}} = \mathbb{P}\parens{G_{<2} \middle| RU_2, K_2}
\end{equation}
Furthermore, by Corollary \ref{cor:keyrand2}, if $G_{<2} = G_1$ holds, the first oracle response is in the span of the first query $x^{(1)} = x'^{(1)} \in X_2 \subseteq X_n$ and $v_{i^{(1)},0} \in K_2 \subseteq K_n$, neither of which is affected by the rotation $R$. Therefore, the second queries, which are a deterministic function of the identical first queries and oracle responses, are equal $x^{(2)} = x'^{(2)}$.

For the inductive step, suppose $\mathbb{P}\parens{G_{<t} \middle| U_t, K_t} = \mathbb{P}\parens{G_{<t} \middle| RU_t, K_t}$ for all $t < n$, and that if these probabilities are one, then $x^{(t)} = x'^{(t)}$.
By the inductive hypothesis, $\mathbb{P}\parens{G_{<n-1} \middle| U_{n-1}, K_{n-1}} = \mathbb{P}\parens{G_{<n-1} \middle| RU_{n-1} K_{n-1}}$. If both probabilities are zero, then $\mathbb{P}\parens{G_{<n} \middle| U_{n}, K_{n}} = 0 = \mathbb{P}\parens{G_{<n} \middle| RU_{n}, K_{n}}$ and we are finished, because $G_{<n} \subseteq G_{<n-1}$. We continue assuming these probabilites are both one.

Again by the inductive hypothesis, if $G_{<n-1}$ holds then $x^{(n-1)} = x'^{(n-1)}$. Thus, the projections $P_{n-1}$ and $P'_{n-1}$ (the projection that arises when $\set{v_{i,r}} = RU_n \cup K_n$) are equal because all of the first $n-1$ queries are identical, and the rotation $R$ does not affect $K_{n-1}$. Consequently,
\begin{equation}
\abs{\inner{\normalized{P'^\perp_{n-1} x'^{(n-1)}}}{Rv}} = \abs{\inner{\frac{R^\top P^\perp_{n-1} x^{(n-1)}}{\norm{P^\perp_{n-1} x^{(n-1)}}}}{v}} =  \abs{\inner{\normalized{P^\perp_{n-1} x^{(n-1)}}}{v}}
\end{equation}
and we conclude $\mathbb{P}\parens{G_{<n} \middle| U_n, K_n} = 1 \iff  \mathbb{P}\parens{G_{<n} \middle| RU_n, K_n} = 1$. Finally, if $G_{<n}$ holds, then by Corollary \ref{cor:keyrand2}, all oracle responses before time $n$ are determined by $\set{x^{(1)} = x'^{(1)},\dots, x^{(n-1)} = x'^{(n-1)}} \cup K_n$, which is unaffected by $R$, so $x^{(n)} = x'^{(n)}$. 

This establishes that $p_{U_n}(U_n | G_{<n}, K_n) = p_{U_n}(RU_n | G_{<n}, K_n)$ and thus the distribution of $U_n$ is invariant to rotations $R$ that preserve $S_n$.

As a consequence, for any $v \in U_n$, $P_n^\perp v$ and $P_n^\perp Rv$ have the same distribution. Since $R$ preserves $S_n$ and length, we conclude $P_n^\perp v$ and $R P_n^\perp v$ also have the same distribution, and thus $\normalized{P_n^\perp v}$ is distributed uniformly on the unit sphere in $S_n^\perp$ conditional on $G_{<n}$ and $K_n$.

We are now ready to lower bound 
\begin{equation}
    \mathbb{P}\left[G_n \middle| G_{<n} \right]
    = \mathbb{E}_{K_n}\left[ \mathbb{P}\left[G_n \middle| G_{<n}, K_n \right] \right]
    \geq \inf_{K_n} \mathbb{P}\left[G_n \middle| G_{<n}, K_n \right]
\end{equation}
For any arbitrary $K_n$,
\begin{align}
\mathbb{P}\left[G_n \middle| G_{<n}, K_n \right]
&= \mathbb{P}\left[\forall v \in U_n \abs{\inner{\normalized{P_n^\perp x^{(n)}}}{v}} \leq \alpha \middle| G_{<n}, K_n \right] \\
&\geq 1 - \sum_{v \in U_n} \mathbb{P}\left[\abs{\inner{\normalized{P_n^\perp x^{(n)}}}{v}} > \alpha \middle| G_{<n}, K_n \right] \\
&\geq 1 - \sum_{v \in U_n} \mathbb{P}\left[\abs{\inner{\normalized{P_n^\perp x^{(n)}}}{\normalized{P_n^\perp v}}} > \alpha \middle| G_{<n}, K_n \right] \label{eq:inner-products-referenced-lem3}
\end{align}
For each inner product \eqref{eq:inner-products-referenced-lem3}, the first term is a fixed quantity conditioned on $G_{<n}$ and $K_n$. The conditional distribution of the second vector, $\normalized{P_n^\perp v}$, as argued above, is uniform on $S_{n}^\perp$. Thus, each term of the sum is the inner product of a fixed unit vector and a uniformly random unit vector in $d - 2(n-1)$ dimensions. 

This probability that the inner product is at least $\alpha$ is proportional to the surface area of the ``end caps'' of a unit sphere lying above and below circles of radius $r = \sqrt{1-\alpha^2}$. This surface area, in turn, is less than the surface area of a sphere of radius $r$. Therefore, for a given $v \in U_n$
\begin{align}
\mathbb{P}\left[\abs{\inner{\normalized{P_n^\perp x^{(n)}}}{\normalized{P_n^\perp v}}} > \alpha \middle| G_{<n}, K_n \right]
&\leq \frac{r^{d - 2n + 2}}{1^{d - 2n + 2}} \\
&= \parens{1-\alpha^2}^{\frac{d - 2n + 2}{2}} \\
&\leq \exp\parens{-\frac{\alpha^2\parens{d - 2n + 2}}{2}}
\end{align}
We conclude that
\begin{align}
\mathbb{P}\left[G_n \middle| G_{<n}, K_n \right]
&\geq 1 - \sum_{v \in U_n}\exp\parens{-\frac{\alpha^2\parens{d - 2n + 2}}{2}} \\
&\geq 1 - \abs{U_n}\exp\parens{-\frac{\alpha^2\parens{d - 2n + 2}}{2}} \\
&\geq 1 - \frac{mk}{2}\exp\parens{-\frac{\alpha^2\parens{d - 2n + 2}}{2}}
\end{align}
This held for an arbitrary $K_n$, thus
\begin{equation}
    \mathbb{P}\left[G_n \middle| G_{<n} \right]\geq 1 - \frac{mk}{2}\exp\parens{-\frac{\alpha^2\parens{d - 2n + 2}}{2}}
\end{equation}
\end{proof}

\begin{lemma}\label{lem:keyrand4}
For any $d \geq 2N + \frac{2}{\alpha^2}\log\parens{2mkN}$
\[
    \mathbb{P}[E] \geq \frac{3}{4}
\]
\end{lemma}
\begin{proof}
By Lemma \ref{lem:keyrand3}, for all $n \leq N$
\begin{equation}
    \mathbb{P}\left[G_n \middle| G_{<n} \right]\geq 1 - \frac{mk}{2}\exp\parens{-\frac{\alpha^2\parens{d - 2n + 2}}{2}}
\end{equation}
Thus,
\begin{align}
\mathbb{P}[E] 
&\geq \mathbb{P}\left[ G_{\leq N} \right]
= \prod_{n\leq N} \mathbb{P}\left[ G_{n} \middle| G_{<n} \right] \\
&\geq \prod_{n\leq N} 1 - \frac{mk}{2}\exp\parens{-\frac{\alpha^2\parens{d - 2n + 2}}{2}} \\
&\geq \parens{1 - \frac{mk}{2}\exp\parens{-\frac{\alpha^2\parens{d - 2N + 2}}{2}}}^N \\
&\geq 1 - \frac{mkN}{2}\exp\parens{-\frac{\alpha^2\parens{d - 2N + 2}}{2}} \\
&\geq 1 - \frac{mkN}{2}\exp\parens{-\log\parens{2mkN}} \\
&\geq 1 - \frac{1}{4}
\end{align}
\end{proof}

Together, Lemma \ref{lem:keyrand4} and Property \ref{keyprop} allow us to ensure that any algorithm can only learn one important vector per query with high probability as long as the dimension is large enough. What is left is to show that Property \ref{keyprop} holds for each of our constructions and to bound the suboptimality of any iterate that has small inner product with the vectors in $U_n$.

\subsection{Non-smooth and not strongly convex components \label{appendix:RandLBLower}}
We first consider the Lipschitz and non-strongly convex setting and prove theorem 5:
\RandLBLower*
Without loss of generality, we let $L = B = 1$. As shown in Equations \eqref{main:defpsi} and \eqref{main:lnsc}, we define
\begin{equation*}
\psi_c(z) = \max\left(0, \abs{z} - c\right)
\end{equation*}
and for values $b$, $c$, and $k$ to be fixed later we define $m/2$ pairs of functions, indexed by $i=1..m/2$:
\begin{align*}
  f_{i,1}(x) &= \frac{1}{\sqrt{2}}\abs{b - \inner{x}{v_{i,0}}} + \frac{1}{2\sqrt{k}} \sum_{r\text{ even}}^k \psi_c\left(\inner{x}{v_{i,r-1}} - \inner{x}{v_{i,r}}\right)    \\
  f_{i,2}(x) &= \frac{1}{2\sqrt{k}} \sum_{r\text{ odd}}^k
  \psi_c\left(\inner{x}{v_{i,r-1}} - \inner{x}{v_{i,r}}\right)
\end{align*}
Assume for now that $m$ is even. If $m$ is odd, then we simply set one of the functions to $0$ and the oracle complexity is reduced by a factor proportional to $\frac{m-1}{m}$.

At the end of this proof, we will show that the functions $f_{i,\cdot}$ satisfy Property \ref{keyprop}. Since the domain, and therefore the queries made to the oracle are bounded by $B=1$, Property \ref{keyprop} and Lemma \ref{lem:keyrand4} ensure that
when the dimension is at least 
$d \geq 2N + \frac{8\parens{1 + \sqrt{2N}}^2}{c^2}\log\parens{2mkN}$,
for iterate $x$ generated after $N$ oracle queries, $\inner{x}{v_{i,r}} \geq \frac{c}{2}$ for no more than $N$ vectors $v_{i,r}$ with probability $\frac{3}{4}$.

We now bound the suboptimality of $(f_{i,1} + f_{i,2})/2$ for any $x$ where $\inner{x}{v_{i,k}} < \frac{c}{2}$. 
\begin{align*}
\frac{1}{2}(f_{i,1}(x) + f_{i,2}(x)) = \frac{1}{2\sqrt{2}}\abs{b - \inner{x}{v_{i,0}}} + \frac{1}{4\sqrt{k}} \sum_{r=1}^k \psi_c\left(\inner{x}{v_{i,r-1}} - \inner{x}{v_{i,r}}\right) \\
\end{align*}
It is straightforward to confirm that this function is minimized when $\inner{x}{v_{i,r}} = b$ for all $r$. Since this is also true for every $i$, $F$ is minimized at $x_b = b\sum_{i=1}^{\frac{m}{2}}\sum_{r=0}^{k}v_{i,r}$. In order that $\norm{x_b} = 1$ so that $x_b \in \mathcal{X}$, we set $b = \sqrt{\frac{2}{m(k+1)}}$. Thus,
\begin{align*}
\frac{1}{2}(f_{i,1}(x) + f_{i,2}(x)) - \frac{1}{2}(f_{i,1}(x^*) + f_{i,2}(x^*)) &\geq \frac{1}{2}(f_{i,1}(x) + f_{i,2}(x)) - 0 \\
&\geq \frac{1}{2\sqrt{2}}\abs{b - \inner{x}{v_{i,0}}} + \frac{1}{4\sqrt{k}} \sum_{r=1}^k \abs{\inner{x}{v_{i,r-1}} - \inner{x}{v_{i,r}}} - c \\
&\geq -\frac{k}{4\sqrt{k}}c + \frac{1}{2\sqrt{2}}\abs{b - \inner{x}{v_{i,0}}} + \frac{1}{4\sqrt{k}}\abs{\inner{x}{v_{i,0}} - \inner{x}{v_{i,k}}} \\
&\geq -\frac{k}{4\sqrt{k}}c  + \frac{1}{2\sqrt{2}}\abs{b - \inner{x}{v_{i,0}}} + \frac{1}{4\sqrt{k}}\abs{\inner{x}{v_{i,0}}} - \frac{1}{4\sqrt{k}}\frac{c}{2} \\
&\geq -\frac{2k + 1}{8\sqrt{k}} c + \min_{z \in \mathbb{R}} \frac{1}{2\sqrt{2}}\abs{b - z} + \frac{1}{4\sqrt{k}}\abs{z} \\
&= -\frac{2k + 1}{8\sqrt{k}}c + \frac{b}{4\sqrt{k}} \\
&\geq -\frac{2k + 1}{8\sqrt{k}}c + \frac{1}{4k\sqrt{m}}
\end{align*}
Therefore, we set $c = \min\set{\frac{1}{\sqrt{N}}, \frac{\epsilon}{\sqrt{k}}}$ and $k = \lfloor \frac{1}{10\epsilon\sqrt{m}} \rfloor$ so that
\[ \frac{1}{2}(f_{i,1}(x) + f_{i,2}(x)) - \frac{1}{2}(f_{i,1}(x^*) + f_{i,2}(x^*)) \geq -\frac{\epsilon}{2} + \frac{5}{2}\epsilon = 2\epsilon \]
This ensures that if $\inner{x}{v_{i,k}} < \frac{c}{2}$ for at least $m/4$ $i$'s, then $x$ cannot be $\epsilon$-suboptimal for $F$. Therefore, after $N = \frac{m(k+1)}{4}$ queries in dimension
\begin{align}
d 
&\geq 2N + \frac{8\parens{1 + \sqrt{2N}}^2}{c^2}\log\parens{2mkN} \\
&= \frac{m(k+1)}{2} + \frac{8\parens{1 + \sqrt{2N}}^2}{c^2}\log\parens{\frac{m^2k(k+1)}{2}} \\
&= \Omega\parens{\parens{\frac{m}{\epsilon^2} + \frac{1}{\epsilon^3\sqrt{m}}}\log\parens{\frac{m}{\epsilon^2}}}
\end{align}
then $F(x) - F(x^*) \geq \epsilon$ with probability $\frac{3}{4}$. When $\epsilon < \frac{1}{10\sqrt{m}}$, $A$ must make at least
\[ \frac{m(k+1)}{4} \geq \frac{m}{4} + \frac{\sqrt{m}}{80\epsilon} \]
queries with probability $\frac{3}{4}$.
Finally, we prove that Property \ref{keyprop} holds for our construction:

\begin{proof}[Proof of Propetry \ref{keyprop} for Lipschitz, non-strongly convex construction]
First we prove the properties about the gradients:

Consider the case when $t$ is odd. From \eqref{main:defpsi}, it is clear that $\frac{d\psi_c}{dz}(z) = 0$ when $\abs{z} < c$. Furthermore, for $r > t$, $\abs{\inner{x}{v_{i,r-1}} - \inner{x}{v_{i,r}}} < c$. Therefore, any subgradient $g_1 \in \partial f_{i,1}(x)$ and $g_2 \in \partial f_{i,2}(x)$ can be expressed as
\begin{align*}
g_1 &= \frac{\textrm{sign}(b - \inner{x}{v_{i,0}})}{\sqrt{2}}v_{i,0} + \frac{1}{2\sqrt{k}}\sum_{r\text{ even}}^{t-1}\psi_c'(\inner{x}{v_{i,r-1}} - \inner{x}{v_{i,r}})(v_{i,r-1} - v_{i,r})   \\
g_2 &=  \frac{1}{2\sqrt{k}}\sum_{r\text{ odd}}^{t}\psi_c'(\inner{x}{v_{i,r-1}} - \inner{x}{v_{i,r}})(v_{i,r-1} - v_{i,r})
\end{align*}
where $\text{sign}(0)$ can take any value in the range $[-1,1]$ and where $\psi_c'$ is a subderivative of $\psi_c$. It is clear from these expressions that $\partial f_{i,1}(x) \subseteq \spn{v_{i,0},...,v_{i,t-1}}$ and $\partial f_{i,2}(x) \subseteq \spn{v_{i,1},...,v_{i,t}}$. The proof for the case when $t$ is even follows the same line of reasoning.

We now prove the properties about the proxs:

Since each pair of functions $f_{i,\cdot}$ operates on a separate $(k+1)$-dimensional subspace, it will be useful to decompose vectors into $x = x_i^v + x_i^\perp$ where $x_i^v = \sum_{r=0}^{k} \inner{x}{v_{i,r}}v_{i,r}$ and $x_i^\perp = x - x_i^v$.
First, note that 
\begin{align*}
\prox_{f_{i,1}}(x,\beta) &= \argmin_{u} f_{i,1}(u) + \frac{\beta}{2}\norm{x - u}^2 \\
&= \argmin_{u_i^v, u_i^\perp} f_{i,1}(u_i^v) + \frac{\beta}{2}\norm{x_i^v + x_i^\perp - u_i^v - u_i^\perp}^2 \\
&= \argmin_{u_i^v, u_i^\perp} f_{i,1}(u_i^v) + \frac{\beta}{2}\norm{x_i^v - u_i^v}^2 + \norm{x_i^\perp - u_i^\perp}^2 \\
&= \argmin_{u_i^v} f_{i,1}(u_i^v) + \frac{\beta}{2}\norm{x_i^v - u_i^v}^2 + \argmin_{u_i^\perp} \frac{\beta}{2}\norm{x_i^\perp - u_i^\perp}^2\\
&= x_i^\perp + \argmin_{u_i^v} f_{i,1}(u_i^v) + \frac{\beta}{2}\norm{x_i^v - u_i^v}^2 \\
&= x_i^\perp + \prox_{f_{i,1}}(x_i^v,\beta)
\end{align*}
(and similarly for $f_{i,2}$). From there, the proof is similar to the proof of lemma \ref{lem:DetLBProxSpan}. First, consider the function $f_{i,2}$ and let $t' \geq t$ be the smallest even number which is not smaller than $t$. It will be convenient to further decompose vectors into $x_i^v = x^- + x^+$ where $x^- = \sum_{r=0}^{t'-1} \inner{x_i^v}{v_{i,r}}v_{i,r}$ and $x^+ = \sum_{r=t'}^{k} \inner{x_i^v}{v_{i,r}}v_{i,r}$. So
\begin{align*}
f_{i,2}(x^- + x^+) &= \frac{1}{2\sqrt{k}} \sum_{r\in \set{1,3,...,t'-1}} \psi_c\left(\inner{x^-}{v_{i,r-1}} - \inner{x^-}{v_{i,r}}\right) \\
& +  \frac{1}{2\sqrt{k}}\sum_{r\in \set{t'+1,t'+3,...k}} \psi_c\left(\inner{x^+}{v_{i,r-1}} - \inner{x^+}{v_{i,r}}\right) \\
&= f_{i,2}(x^-) + f_{i,2}(x^+)
\end{align*}
Therefore,
\begin{align*}
\prox_{f_{i,2}}(x_i^v,\beta) &= \argmin_{u^-,u^+} f_{i,2}(u^- + u^+) + \frac{\beta}{2}\norm{x^- + x^+ - u^- - u^+}^2 \\
&=  \argmin_{u^-,u^+} f_{i,2}(u^- + u^+) + \frac{\beta}{2}\norm{x^- - u^-}^2 + \frac{\beta}{2} \norm{x^+ - u^+}^2 \\
&= \argmin_{u^-,u^+} f_{i,2}(u^-) + f_{i,2}(u^+) + \frac{\beta}{2}\norm{x^- - u^-}^2 + \frac{\beta}{2} \norm{x^+ - u^+}^2 \\
&= \argmin_{u^-} f_{i,2}(u^-) + \frac{\beta}{2}\norm{x^- - u^-}^2  + \argmin_{u^+}  f_{i,2}(u^+)  +\frac{\beta}{2} \norm{x^+ - u^+}^2
\end{align*}
Since $\abs{\inner{x^+}{v_{i,r}}} < \frac{c}{2}$ for all $r \geq t$, $\abs{\inner{x^+}{v_{i,r-1}} - \inner{x^+}{v_{i,r}}} < c$ for $r > t$, which implies that $f_{i,2}(x^+) = 0$. Therefore, the objective of the second $\argmin$ is non-negative and is equal to zero when $u^+ = x^+$ so
\begin{align*}
\prox_{f_{i,2}}(x_i^v,\beta) &= x^+ + \argmin_{u^-} f_{i,2}(u^-) + \frac{\beta}{2}\norm{x^- - u^-}^2  
\end{align*}
Therefore, when $t$ is even, $t' = t$ and $\prox_{f_{i,2}}(x,\beta) \in \spn{x, v_{i,0},...,v_{i,t-1}}$, and when $t$ is odd, $t' = t+1$ and $\prox_{f_{i,2}}(x,\beta) \in \spn{x,v_{i,0},...,v_{i,t}}$. A very similar line of reasoning can be used to show the statement for $f_{i,1}$.
\end{proof}

\paragraph{Remark:} As was mentioned before, Lemma \ref{lem:keyrand4} applies when the norm of every query point is bounded by $B$. Since all points in the domain of the optimization problem have norm bounded by $B$, this is not problematic. However, we can slightly modify our construction to make optimizing $F$ hard \emph{even} for algorithms that are allowed to query outside of the domain. 

We could redefine our functions as follows:
\[ f'_{i,j}(x) = \begin{cases} f_{i,j}(x) & \norm{x} \leq B \\ f_{i,j}\left(B\frac{x}{\norm{x}}\right) + L\left( \norm{x} - B \right) & \norm{x} > B \end{cases} \]
$f'_{i,j}$ is still continuous, and $L$-Lipschitz, and it also has the property that it behaves exactly like $f_{i,j}$ on $B$-ball. However, querying the oracle of $f'_{i,j}$ outside of the $B$-ball gives no more information about the function than querying at $B\frac{x}{\norm{x}}$. In fact, an algorithm that was only allowed to query within the $B$-ball would be able to simulate the oracle of $F'$. Therefore, since the algorithm that is not allowed to query at large vectors cannot optimize $F'$ quickly, and it could simulate queries with unbounded norm, it follows that querying with unbounded norm cannot improve the rate of convergence. This fact is needed in the proof of Theorem \ref{thm:RandLscLower} below.

\subsection{Non-smooth and strongly convex components \label{appendix:RandLscLower}}
We now prove Theorem \ref{thm:RandLscLower} using a reduction from the Lipschitz and \emph{non}-strongly convex setting:
\RandLscLower*
\begin{proof}
Just as in the proof of Theorem \ref{thm:DetLscLower}, we assume towards contradiction that there is an algorithm $A$ which can optimize $F$ using $o\left(m + \frac{\sqrt{m}L}{\sqrt{\lambda\epsilon}} \right)$ queries to $h_F$ in expectation. Then $A$ could be used to minimize the sum $\tilde{F}$ of $m$ functions $\tilde{f}_i$, which are convex and $L$-Lipschitz continuous over the domain $\set{x:\norm{x}\leq B}$ by adding a regularizer. Let
\[ F(x) = \frac{1}{m} \sum_{i=1}^m f_i(x) := \frac{1}{m} \sum_{i=1}^m \tilde{f}_i(x) + \frac{\lambda}{2}\norm{x}^2  \]
Note that $f_i$ is $\lambda$-strongly convex and since $\tilde{f}_i$ is $L$-Lipschitz, $f_i$ is $(L + \lambda B)$-Lipschitz continuous on the same domain. Furthermore, by setting $\lambda = \frac{\epsilon}{B^2}$, 
\[ \tilde{F}(x) \leq F(x) \leq \tilde{F}(x) + \frac{\epsilon}{2B^2}\norm{x}^2 \leq \tilde{F}(x) + \frac{\epsilon}{2} \]
By assumption, $A$ can find an $\hat{x}$ such that $F(\hat{x}) - F(x^*) < \frac{\epsilon}{2}$ using $o\left( m + \frac{\sqrt{m}(L + \lambda B)}{\sqrt{\lambda\epsilon}} \right) = o\left( m + \frac{\sqrt{m}LB}{\epsilon} \right)$ queries to $h_F$, and
\[ \frac{\epsilon}{2} > F(\hat{x}) - F(x^*) \geq \tilde{F}(\hat{x}) - \tilde{F}(\tilde{x}^*) - \frac{\epsilon}{2}  \]
Thus $\hat{x}$ is $\epsilon$-suboptimal for $\tilde{F}$. However, this contradicts the conclusion of theorem \ref{thm:RandLBLower} when $L > 0$, $\lambda > 0$, $0 < \epsilon < \frac{L^2}{200\lambda m}$, and $d = \Omega\left( \frac{L^3}{\sqrt{\lambda^3\epsilon^3 m}}\log \frac{L^2m}{\lambda\epsilon} \right)$ leads to contradiction.
\end{proof}

\subsection{Smooth and not strongly convex components \label{appendix:RandSBLower}}
\RandSBLower*
Without loss of generality, we can assume that $\gamma = B = 1$. We will first consider the case where $\epsilon = O\left( \frac{1}{m} \right)$ and prove that $A$ must make $\Omega\left( \sqrt{\frac{m}{\epsilon}} \right)$ queries to $h_F$. Afterwards, we will show a lower bound of $\Omega(m)$ in the large-$\epsilon$ regime where that term dominates.

The function construction in this case is very similar to the non-smooth randomized construction. As in Equation \eqref{main:defphi}
\begin{equation*}
\phi_c(z) = \begin{cases} 0 & \abs{z} \leq c \\ 2(\abs{z} - c)^2 & c < \abs{z} \leq 2c \\ z^2 - 2c^2 & \abs{z} > 2c \end{cases}
\end{equation*}
The key properties of this function for this proof are that it is convex, everywhere differentiable and $4$-smooth, and when $\abs{z} \leq c$, the function is constant at 0. It is also useful to note that 
\begin{equation} \label{eq:phifunction}
0 \leq z^2 - \phi_c(z) \leq 2c^2 
\end{equation}
As in Equation \eqref{main:snsc}, for values $a$ and $k$ to be fixed later, we define the pairs of functions for $i = 1,...,m/2$: 
\begin{align*} 
f_{i,1}(x) &= \frac{1}{16}\left( \inner{x}{v_{i,0}}^2 - 2a\inner{x}{v_{i,0}} + \sum_{r \in \set{2,4,...} \leq k} \phi_c\left(\inner{x}{v_{i,r-1}} - \inner{x}{v_{i,r}}\right)  \right) \\
f_{i,2}(x) &= \frac{1}{16}\left( \sum_{r \in \set{1,3,...} \leq k} \phi_c\left(\inner{x}{v_{i,r-1}} - \inner{x}{v_{i,r}}\right) + \phi_c\left(\inner{x}{v_{i,k}}\right) \right)
\end{align*}
with orthonormal vectors $v_{i,r}$ chosen randomly on the unit sphere in $\mathbb{R}^d$ as for Theorem \ref{thm:RandLBLower}. 

At the end of this proof, we will show that the functions $f_{i,\cdot}$ satisfy Property \ref{keyprop}. Since the domain, and therefore the queries made to the oracle are bounded by $B$, Property \ref{keyprop} and Lemma \ref{lem:keyrand4} ensure that when the dimension is at least 
$d = 2N + \frac{8(1+\sqrt{2N})^2}{c^2}\log(2mkN)$
then after $N$ oracle queries, $\inner{x}{v_{i,r}} \geq \frac{c}{2}$ for no more than $N$ vectors $v_{i,r}$ with probability $\frac{3}{4}$.


Now, we will bound the suboptimality of $F_i(x) := (f_{i,1}(x) + f_{i,2}(x)) / 2$ at an iterate $x$ such that $\abs{\inner{x}{v_{i,r}}} < \frac{c}{2}$ for all $r \geq t$. From the definition of $\phi_c$:
\begin{align*}
F_i(x) &= \frac{1}{32}\left( \inner{x}{v_{i,0}}^2 - 2a\inner{x}{v_{i,0}} + \sum_{r=1}^k \phi_c\left(\inner{x}{v_{i,r-1}} - \inner{x}{v_{i,r}}\right) + \phi_c\left(\inner{x}{v_{i,k}}\right)\right)\\
&= \frac{1}{32}\left( \inner{x}{v_{i,0}}^2 - 2a\inner{x}{v_{i,0}} + \sum_{r=1}^{t} \phi_c\left(\inner{x}{v_{i,r-1}} - \inner{x}{v_{i,r}}\right) \right) \\
F_i(x) &\leq \frac{1}{32}\left( \inner{x}{v_{i,0}}^2 - 2a\inner{x}{v_{i,0}} + \sum_{r=1}^{t} \left(\inner{x}{v_{i,r-1}} - \inner{x}{v_{i,r}}\right)^2 + \inner{x}{v_{i,t}}^2\right)\\
F_i(x) &\geq \frac{1}{32}\left( \inner{x}{v_{i,0}}^2 - 2a\inner{x}{v_{i,0}} + \sum_{r=1}^{t} \left(\inner{x}{v_{i,r-1}} - \inner{x}{v_{i,r}}\right)^2 + \inner{x}{v_{i,t}}^2\right)-\frac{t+1}{16}c^2
\end{align*}
Define 
\[ F_i^{t+1}(x) := \frac{1}{32}\left( \inner{x}{v_{i,0}}^2 - 2a\inner{x}{v_{i,0}} + \sum_{r=1}^{t} \left(\inner{x}{v_{i,r-1}} - \inner{x}{v_{i,r}}\right)^2 + \inner{x}{v_{i,{t}}}^2\right) \]
and note that in the proof of Theorem \ref{thm:DetSBLower} we already showed that that the optimum of $F^t_i$ is achieved at 
\[ x_{i,t}^* = a\sum_{r=0}^{t-1} \left(1 - \frac{r + 1}{t+1}\right)v_{i,r} \]
and 
\[ F_i^t\left(x_{i,t}^*\right) = -\frac{a^2}{32}\left(1 - \frac{1}{t+1}\right) \]
and
\[ \norm{x_{i,t}^*}^2 \leq \frac{a^2t}{3} \]
Therefore, setting $a = \sqrt{\frac{6}{m(k+1)}}$ ensures that $\norm{\sum_{i=1}^{\frac{m}{2}}x_{i,k+1}^*} \leq 1$. It is not necessarily true that $x^* = \sum_{i=1}^{\frac{m}{2}}x_{i,k+1}^*$, but it serves as an upper bound on the optimum.

Let $q := \lfloor\frac{k}{2}\rfloor$ and consider an iterate $x$ generated by $A$ before it makes $q - 1$ queries to the functions $f_{i,1}$ and $f_{i,2}$. When $\inner{x}{v_{i,r}} < \frac{c}{2}$ for all $r \geq q$,
\begin{align*}
F_i(x) - F_i(x^*) &\geq F_i^q(x) - \frac{qc^2}{16} - F_i(x_{i,k+1}^*) \\
&\geq F_i^q(x_{i,q}^*) - F_i^{k+1}(x_{i,k+1}^*) - \frac{qc^2}{16} \\
&= -\frac{a^2}{32}\left(1 - \frac{1}{q+1}\right) + \frac{a^2}{32}\left(1 - \frac{1}{k+2}\right) - \frac{qc^2}{16} \\
&\geq \frac{1}{32k^2m} - \frac{kc^2}{32}
\end{align*}
where the last inequality holds as long as $k \geq 2$. When $\epsilon < \frac{1}{320m}$, setting $c = \min\set{\frac{1}{\sqrt{N}}, \sqrt{\frac{16\epsilon}{k}}}$ and $k = \lfloor \frac{1}{\sqrt{80\epsilon m}} \rfloor \geq 2$, ensures that
\[ F_i(x) - F_i(x^*) \geq \frac{5}{2}\epsilon - \frac{\epsilon}{2} = 2\epsilon \]
Therefore, if $\inner{x}{v_{i,r}} < \frac{c}{2}$ for all $r \geq q$ is true for at least $\frac{m}{4}$ of the $i$'s, then $x$ cannot be $\epsilon$-suboptimal for $F$. So, for $N = \frac{mq}{4}$ in dimension
\begin{align}
d 
&\geq 2N + \frac{8(1+\sqrt{2N})^2}{c^2}\log(2mkN) \\
&= \frac{mq}{2} + \frac{8(1+\sqrt{2N})^2}{c^2}\log\parens{\frac{m^2kq}{2}} \\
&= \Omega\parens{\parens{\sqrt{\frac{m}{\epsilon}} + \frac{1}{\epsilon^2}}\log\parens{\frac{m}{\epsilon}}}
\end{align}
 with probability $\frac{3}{4}$, the algorithm must make at least $\frac{mq}{4}$ queries in order to reach an $\epsilon$-suboptimal point. This gives a lower bound of
\[ \frac{m}{4}q \geq \frac{\sqrt{m}}{48\sqrt{10\epsilon}} \]
which holds with probability $\frac{3}{4}$. To complete the first half of the proof, we prove that Property \ref{keyprop} holds for this construction:

\begin{proof}[Proof of Property \ref{keyprop} for smooth and non-strongly convex construction]
First we prove the properties about gradients:

Consider the case when $t$ is odd. From equation \ref{eq:phifunction}, we can see that $\frac{d\phi_c}{dz}(z) = 0$ when $\abs{z} < c$. Furthermore, for $r > t$, $\abs{\inner{x}{v_{i,r-1}} - \inner{x}{v_{i,r}}} < c$. We can therefore express the gradients:
\begin{align*}
\nabla f_{i,1}(x)&= \frac{1}{16}\left( 2\inner{x}{v_{i,0}} - 2av_{i,0} + \sum_{r\text{ even}}^{t-1} \phi'_c\left(\inner{x}{v_{i,r-1}} - \inner{x}{v_{i,r}}\right)(v_{i,r-1}-v_{i,r})  \right)   \\
\nabla f_{i,2}(x) &= \frac{1}{16}\left( \sum_{r\text{ odd}}^t \phi'_c\left(\inner{x}{v_{i,r-1}} - \inner{x}{v_{i,r}}\right)(v_{i,r-1} - v_{i,r}) + \phi'_c\left(\inner{x}{v_{i,k}}\right)v_{i,k} \right)
\end{align*}
It is clear from these expressions that $\nabla f_{i,1}(x) \in \spn{v_{i,0},...,v_{i,t-1}}$ and $\nabla f_{i,2}(x) \in \spn{v_{i,0},...,v_{i,t}}$. The proof for the case when $t$ is even follows the same line of reasoning.

Now, we prove the properties about proxs:

We follow the same line of reasoning as in the Lipschitz and non-strongly convex case. The only necessary addition is to show, that when $t'\geq t$ is the smallest even number which is not smaller than $t$ and $u^- = \sum_{r=0}^{t'-1} \inner{u_i^v}{v_{i,r}}v_{i,r}$ and $u^+ = \sum_{r=t'}^{k} \inner{u_i^v}{v_{i,r}}v_{i,r}$, then $f_{i,2}(u^- + u^+) = f_{i,2}(u^-) + f_{i,2}(u^+)$: 
\begin{align*}
f_{i,2}(u^- + u^+) &= \frac{1}{16}\bigg(  \sum_{r \in \set{1,3,...,t'-1}} \phi_c\left(\inner{u^-}{v_{i,r-1}} - \inner{u^-}{v_{i,r}}\right) \\
&  + \sum_{r \in \set{t'+1,t'+3,...} < k} \phi_c\left(\inner{u^+}{v_{i,r-1}} - \inner{u^+}{v_{i,r}}\right) + \phi_c\left(\inner{u^+}{v_{i,k}}\right) \bigg) \\
&= f_{i,2}(u^-) + f_{i,2}(u^+)
\end{align*} 
This same reasoning applies for $f_{i,1}$ or odd $t'$. 
\end{proof}

So far, we have shown a lower bound of $\Omega\left( \sqrt{\frac{m}{\epsilon}} \right)$ when $\epsilon = O\left(\frac{1}{m}\right)$. We now show a lower bound of $\Omega(m)$ for all $\epsilon > 0$, which accounts for the first term in the lower bound $\Omega\left(m + \sqrt{\frac{m}{\epsilon}}\right)$ which dominates when $\epsilon = \Omega\left(\frac{1}{m}\right)$. Consider the $0$-smooth functions
\[ f_i(x) = C\inner{x}{v_i} \]
for any constant $C > 0$, and where the orthonormal vectors $v_i$ are randomly chosen as before. $F$ reaches its minimum on the unit ball at
\[ \argmin_{x : \norm{x} \leq 1} F(x) = \frac{-1}{\sqrt{m}}\sum_{i=1}^m v_i \]
and $F(x^*) = -\frac{C}{\sqrt{m}}$. Using similar analysis as inside the proof of Lemma \ref{lem:keyrand3}, if $d = \frac{2B^2}{\left(\frac{1}{4\sqrt{m}}\right)^2}\log 2m \leq 32B^2m\log 2m$ then $\mathbb{P}\left(\exists i\text{ which has not been queried s.t. }\abs{\inner{x}{v_i}} \geq \frac{1}{4\sqrt{m}}\right) < \frac{1}{2}$. So if fewer than $\frac{m}{2}$ functions have been queried, then with probability at least $\frac{1}{2}$:
\[ F(x) - F(x^*) \geq \left(\frac{-C\sqrt{31}}{8\sqrt{m}} + \frac{-C}{8\sqrt{m}}\right) - \frac{-C}{\sqrt{m}} \geq \frac{0.16C}{\sqrt{m}} \]
so 
\[ \mathbb{E}\left[ F(x) - F(x^*) \right] \geq \frac{0.08C}{\sqrt{m}} \]
Therefore, by simply choosing $C = \frac{\epsilon\sqrt{m}}{0.08}$, we ensure that such a point $x$ is at least $\epsilon$-suboptimal, completing the proof for all $\epsilon > 0$.

As noted above, the queries made to the oracle must be bounded for Lemma \ref{lem:keyrand4}. Since the domain of $F$ is the $B$-ball, this is easy to satisfy. If we want to ensure that our construction is still hard to optimize, \emph{even} if the algorithm is allowed to query arbitrarily large vectors, then we can modify our construction by defining a new function $f'_{i,j}$ in terms of its gradient
\[ \nabla f'_{i,j}(x) = \begin{cases} \nabla f_{i,j}(x) & \norm{x} \leq B \\ \nabla f_{i,j}\left( B\frac{x}{\norm{x}} \right) & \norm{x} > B \end{cases} \]
This function is continuous and smooth, and also has the property that querying the oracle at a point $x$ outside of the $B$-ball is cannot be more informative than querying at $B\frac{x}{\norm{x}}$. That is, an algorithm that is not allowed to query outside the $B$-ball can simulate such queries using its restricted oracle. Since this restricted algorithm cannot optimize quickly, but can still calculate the oracle outputs that it would have recieved by querying large vectors, it follows that an unrestricted algorithm could not optimize this function quickly either.

\paragraph{Remark:}
Another variant of \eqref{eq:main} that one might consider is an unconstrained optimization problem, where we assume that the minimizer of $F$ lies on the interior of that ball. In other words, we could consider a version of \eqref{eq:main} where the gradient of $F$ must vanish on the interior of $\mathcal{X}$. 

In this case, there is little reason to consider any $\epsilon$ larger than $\frac{\gamma B^2}{2}$, since $F(0) - F(x^*) \leq \frac{\gamma B^2}{2}$ always (by smoothness $F(0) - F(x^*) \leq \inner{\nabla F(x^*)}{x_0 - x^*} + \frac{\gamma}{2}\norm{x^*}^2 \leq \frac{\gamma B^2}{2}$). Consequently, when $\epsilon \geq \frac{\gamma B^2}{2}$ there is a trivial upper bound of \emph{zero} oracle queries, as just returning the zero vector guarantees $\epsilon$-suboptimality. We can construct functions so that Theorem \ref{thm:RandSBLower} still applies for $0 < \epsilon < \frac{9\gamma B^2}{128}$.
In the previous proof, the first construction is still valid in the unconstrained case since the minimizer lies within the unit ball. For the $\Omega(m)$ term, consider the $1$-smooth functions (assume w.l.o.g. that $\gamma = B = 1$)
\[ f_i(x) = \sqrt{m}\inner{x}{v_i} + \frac{\norm{x}^2}{2} \]
where the $m$ orthonormal vectors $v_i$ are drawn randomly from the unit sphere in $\mathbb{R}^d$ as in the previous construction. The gradient of $F$ vanishes at $x^* = -\frac{1}{\sqrt{m}}\sum_{i=1}^mv_i$, (note $\norm{x^*} = 1$) and $F(x^*) = -\frac{1}{2}$. Using similar techniques as inside the proof of Lemma \ref{lem:keyrand3} if $d = \frac{2B^2}{\left(\frac{1}{4\sqrt{m}}\right)^2}\log 10m \leq 32B^2m\log 10m$, then for any iterate $x$ generated by $A$ before $f_i$ has been queried, $\mathbb{P}\left(\exists i\text{ which has not been queried s.t. }\abs{\inner{x}{v_i}} \geq \frac{1}{4\sqrt{m}}\right) < \frac{9}{10}$. Furthermore, if $\abs{\inner{x}{v_i}} < \frac{1}{4\sqrt{m}}$ for more than $\frac{m}{2}$ of the functions, then
\begin{align*}
F(x) &= \frac{1}{m}\sum_{i=1}^m \sqrt{m}\inner{x}{v_i} + \frac{\norm{x}^2}{2} \\
&\geq \frac{1}{m}\left(\frac{m}{2}\cdot\sqrt{m}\frac{-1}{\sqrt{m}} + \frac{m}{2} \cdot \sqrt{m}\frac{-1}{4\sqrt{m}} \right) + \frac{\frac{m}{2} \cdot \frac{1}{m} + \frac{m}{2} \cdot \frac{1}{16m}}{2} \\
&= \frac{-23}{64}\\ 
\end{align*}
Therefore, if fewer than $\frac{m}{2}$ functions have been queried, then with probability at least $\frac{9}{10}$:
\[ F(x) - F(x^*) \geq  \frac{-23}{64} - \frac{-1}{2} = \frac{9}{64} \]
so
\[ \mathbb{E}\left[ F(x) - F(x^*) \right] \geq \frac{9}{128} \]
This proves a lower bound of $\Omega(m)$ for $0 < \epsilon < \frac{9\gamma B^2}{128}$.

\subsection{Smooth and strongly convex components \label{appendix:RandSscLower}}
In the smooth and strongly convex case, we cannot use the same simple reduction that was used to prove Theorem \ref{thm:RandLscLower}. Using that construction, we would be able to show a lower bound of $m$, but would not be able to show any dependence on $\epsilon$, so the lower bound would be loose. Instead, we will use an explicit construction similar to the one used in Theorem \ref{thm:RandSBLower}.
\RandSscLower*
\begin{proof}
We will prove the theorem for a $1$-smooth, $\lambda$-strongly convex problem, for $\lambda < \frac{1}{73m}$, which can be generalized by scaling. 

As in the proof for the non-strongly convex case, we introduce the 4-smooth helper function
\[ \phi_c(z) = \begin{cases} 0 & \abs{z} \leq c \\ 2(\abs{z} - c)^2 & c < \abs{z} \leq 2c \\ z^2 - 2c^2 & \abs{z} > 2c \end{cases} \]
using which we will construct $m/2$ pairs of functions, which will each be based on the following. As in previous proofs, we randomly select orthonormal vectors $v_{i,r}$ from $\mathbb{R}^d$. Then, for constants $k$, $C$, and $\zeta$ to be decided upon later; with $\tilde{\lambda} := m\cdot\lambda$; and for $i = 1,...,\lfloor m/2 \rfloor$ define the following pairs of functions (if $m$ is odd, let $f_m(x) = \frac{\tilde{\lambda}}{2m}\norm{x}^2$):
\begin{align*} 
f_{i,1}(x) &= \frac{1-\tilde{\lambda}}{16}\left( \inner{x}{v_{i,0}}^2 - 2C\inner{x}{v_{i,0}} \sum_{r\text{ even}}^k \phi_c\left(\inner{x}{v_{i,r-1}} - \inner{x}{v_{i,r}}\right) \right) + \frac{\tilde{\lambda}}{2m}\norm{x}^2 \\
f_{i,2}(x) &= \frac{1-\tilde{\lambda}}{16}\left( \zeta\phi_c(\inner{x}{v_{i,k}})+ \sum_{r\text{ odd}}^k \phi_c\left(\inner{x}{v_{i,r-1}} - \inner{x}{v_{i,r}}\right) \right) + \frac{\tilde{\lambda}}{2m}\norm{x}^2
\end{align*}
When $\tilde{\lambda} \in [0,1]$ these function are 1-smooth and $\lambda$-strongly convex. 

These functions also have Property \ref{keyprop}, but we will omit the proof, as it follows directly from the proof in Appendix \ref{appendix:RandSBLower}. Intuitively, the squared norm reveals no new information about the vectors $v_{i,r}$ besides what is already included in the query point $x$.

When all of the queries are bounded by $B$, Property \ref{keyprop} along with Lemma \ref{lem:keyrand4} ensures that when 
$d = 2N + \frac{8B^2\parens{1 + \sqrt{2N}}^2}{c^2}\log\parens{2mkN}$,
after the algorithm make $N$ queries $\inner{x}{v_{i,r}} \geq \frac{c}{2}$ for at most $N$ of the vectors $v_{i,r}$ with probability $\frac{3}{4}$.
For this to apply, we need that all of the queries made by the algorithm are within a $B$-ball around the origin.
We know that $F(0) - F(x^*) = \epsilon_0$, and by strong-convexity $F(0) \geq F(x^*) + \frac{\lambda}{2}\norm{x^*}^2$, therefore, $\norm{x^*} \leq \sqrt{\frac{2\epsilon_0}{\lambda}} =: B$. Since the optimum point must lie in the $B$-ball around the origin, we will restrict the algorithm to query only at points within the $B$-ball. At the end of the proof, we will show that with a small modification to the functions \emph{outside} of the $B$-ball, querying at vectors of large norm cannot help the algorithm.

Now it remains to lower bound the suboptimality of the pair $f_{i,1}$ and $f_{i,2}$ at an iterate which is nearly orthogonal to all vectors $v_{i,r}$ for $r > t$:

In order to bound the suboptimality of a pair of functions $i$, it will be convenient to bundle up all of the terms which affect the value of $\inner{x^*}{v_{i,r}}$ from all $m$ of the component functions. Most of those terms are contained in $f_{i,1}$ and $f_{i,2}$, however, $\norm{x}^2$ terms in each of the other components also affect the value of $\inner{x^*}{v_{i,r}}$. For each $i$, consider the projection operator $P_i$ which projects a vector $x$ onto the subspace spanned by $\set{v_{i,r}}_{r=0}^k$, and $P_\perp$ projecting onto the space orthogonal to $v_{i,r}$ for all $i,r$. Now decompose
\[ \frac{\tilde{\lambda}}{2m}\norm{x}^2 = \frac{\tilde{\lambda}}{2m}\left(\sum_{i=1}^{\lfloor\frac{m}{2}\rfloor} \norm{P_ix}^2 + \norm{P_\perp x}^2 \right)\]
Gather all $m$ of the $\frac{\tilde{\lambda}}{2m}\norm{P_ix}^2$ terms and split them amongst $f_{i,1}$ and $f_{i,2}$ to make the following modified functions:
\[ \tilde{f}_{i,1}(x) = f_{i,1} - \frac{\tilde{\lambda}}{2m}\norm{x}^2 + \frac{\tilde{\lambda}}{4}\norm{P_ix}^2 \]
\[ \tilde{f}_{i,2}(x) = f_{i,2} - \frac{\tilde{\lambda}}{2m}\norm{x}^2 + \frac{\tilde{\lambda}}{4}\norm{P_ix}^2 \]
After this shuffle, all of the terms affecting $\inner{x^*}{v_{i,r}}$ are contained in these two functions which will help the analysis. Note that there is also a remaining $\frac{\tilde{\lambda}}{2}\norm{P_\perp x}^2$ term, however, this term is not very important to track since we are bounding the suboptimality of $F$, which can only increase by considering that non-negative term and $P_\perp x^* = \vec{0}$. Now, consider
\begin{align*}
\frac{1}{2}\left(\tilde{f}_{i,1}(x) + \tilde{f}_{i,2}(x)\right) = \frac{1-\tilde{\lambda}}{32}\bigg( \inner{x}{v_{i,0}}^2 - &2C\inner{x}{v_{i,0}} + \zeta\phi_c(\inner{x}{v_{i,k}}) \\
&+ \sum_{r=1}^k \phi_c\left(\inner{x}{v_{i,r-1}} - \inner{x}{v_{i,r}}\right) \bigg) + \frac{\tilde{\lambda}}{4}\norm{P_ix}^2 
\end{align*}
If we define
\begin{align*}
F_i^t(x) := \frac{1-\tilde{\lambda}}{32}\bigg( \inner{x}{v_{i,0}}^2 - &2C\inner{x}{v_{i,0}} + \inner{x}{v_{i,t}}^2 + \sum_{r=1}^t \left(\inner{x}{v_{i,r-1}} - \inner{x}{v_{i,r}}\right)^2 \bigg) + \frac{\tilde{\lambda}}{4}\norm{P_ix}^2 
\end{align*}
and 
\begin{align*}
F_i(x) := \frac{1-\tilde{\lambda}}{32} \bigg( \inner{x}{v_{i,0}}^2 - &2C\inner{x}{v_{i,0}} + \zeta\inner{x}{v_{i,k}}^2 + \sum_{r=1}^k \left(\inner{x}{v_{i,r-1}} - \inner{x}{v_{i,r}}\right)^2 \bigg) + \frac{\tilde{\lambda}}{4}\norm{P_ix}^2 
\end{align*}
then when $\abs{\inner{x}{v_{i,r}}} < \frac{c}{2}$
\[ F_i^t(x) \leq \frac{1}{2}\left(\tilde{f}_{i,1}(x) + \tilde{f}_{i,2}(x)\right) + \frac{(1-\tilde{\lambda})(t+1)}{16}c^2 \]
and for any $y$
\[ \frac{1}{2}\left(\tilde{f}_{i,1}(y) + \tilde{f}_{i,2}(y)\right) \leq F_i(y) \]

and, conveniently, $F_i$ is very similar to the construction from Appendix \ref{appendix:DetSscLower}. In particular, let $\tilde{Q} := \frac{1}{2}(\frac{1}{\tilde{\lambda}} - 1) + 1$, then
\begin{align*}
F_i^t(x) = \frac{1}{2}\bigg( \frac{\tilde{\lambda}(\tilde{Q}-1)}{8} \bigg( \inner{x}{v_{i,0}}^2 - &2C\inner{x}{v_{i,0}} + \inner{x}{v_{i,t}}^2 + \sum_{r=1}^t \left(\inner{x}{v_{i,r-1}} - \inner{x}{v_{i,r}}\right)^2 \bigg) + \frac{\tilde{\lambda}}{2}\norm{P_ix}^2 \bigg) 
\end{align*}
\begin{align*}
F_i(x) = \frac{1}{2}\bigg( \frac{\tilde{\lambda}(\tilde{Q}-1)}{8} \bigg( \inner{x}{v_{i,0}}^2 - &2C\inner{x}{v_{i,0}} + \zeta\inner{x}{v_{i,k}}^2 + \sum_{r=1}^k \left(\inner{x}{v_{i,r-1}} - \inner{x}{v_{i,r}}\right)^2 \bigg) + \frac{\tilde{\lambda}}{2}\norm{P_ix}^2 \bigg)
\end{align*}
We have already showed in Appendix \ref{appendix:DetSscLower} that if $\hat{x} := \argmin_x F_i(x)$, and if
\begin{align*}
C > \frac{12\sqrt{\epsilon}}{\tilde{\lambda}(\sqrt{\tilde{Q}}-1)} \implies 2\epsilon^i_0 &:= 2\left(F_i(0) - F_i(\hat{x})\right) > \frac{30\epsilon}{\tilde{\lambda}} \\
\zeta &= \frac{2}{\sqrt{\tilde{Q}}+1} \\
\tilde{\lambda} &< \frac{1}{73} \\
t &= \left\lfloor \frac{\sqrt{\tilde{Q}}-1}{4}\log\frac{\epsilon^i_0}{20\sqrt{\tilde{Q}}\epsilon} \right\rfloor \\
\abs{\inner{x}{v_{i,r}}} &\leq \frac{c}{2} \ \ \forall r > t
\end{align*} 
then,
\[ 2\left(F_i^t(x) - F_i(\hat{x})\right) \geq 10\epsilon \]
Therefore,
\begin{align*}
10\epsilon &\leq 2\left(F_i^t(x) - F_i(\hat{x})\right) \\
&\leq 2\left(\frac{1}{2}\left(\tilde{f}_{i,1}(x) + \tilde{f}_{i,2}(x)\right) + \frac{(1-\tilde{\lambda})(k + \zeta)}{16}c^2  - \frac{1}{2}\left(\tilde{f}_{i,1}(\hat{x}) + \tilde{f}_{i,2}(\hat{x})\right)\right) \\
&\leq \left(\tilde{f}_{i,1}(x) + \tilde{f}_{i,2}(x)\right) + \frac{(1-\tilde{\lambda})(k + \zeta)}{8}c^2  - \left(\tilde{f}_{i,1}(x^*) + \tilde{f}_{i,2}(x^*)\right) \\
\end{align*}
So
\[ \left(\tilde{f}_{i,1}(x) + \tilde{f}_{i,2}(x)\right) - \left(\tilde{f}_{i,1}(x^*) -  \tilde{f}_{i,2}(x^*)\right) \geq 10\epsilon - \frac{(1-\tilde{\lambda})(k + \zeta)}{8}c^2 \]
Setting 
\[ c = \min\set{\frac{1}{\sqrt{N}}, \sqrt{\frac{16\epsilon}{(1-\tilde{\lambda})(k + \zeta)}}} \]
then
\[ \left(\tilde{f}_{i,1}(x) + \tilde{f}_{i,2}(x)\right) - \left(\tilde{f}_{i,1}(x^*) -  \tilde{f}_{i,2}(x^*)\right) \geq 10\epsilon - \frac{(1-\tilde{\lambda})(k + \zeta)}{8}c^2 = 8\epsilon \]

Therefore, if at least $m/4$ of the pairs $i$ it holds that $\abs{\inner{x}{v_{i,r}}} < \frac{c}{2}$ for $r > t$, then
\begin{align*} 
F(x) - F(x^*) &\geq \frac{1}{m} \sum_{i=1}^{\lfloor\frac{m}{2}\rfloor} \left(\tilde{f}_{i,1}(x) + \tilde{f}_{i,2}(x)\right) - \left(\tilde{f}_{i,1}(x^*) -  \tilde{f}_{i,2}(x^*)\right) \\
&\geq \frac{m}{4}\cdot\frac{1}{m}\cdot8\epsilon\\
&= 2\epsilon
\end{align*}

As a consequence of this, when the dimension is sufficiently large,
with probability $\frac{3}{4}$ the optimization algorithm must make at least $t$ queries to each of at least $\frac{m}{4}$ pairs of functions in order to reach an $\epsilon$-suboptimal solution in expectation. So, when
\begin{align*}
\lambda &\leq \frac{1}{161m}\\
\frac{\epsilon_0}{\epsilon} &\geq \frac{60}{\sqrt{m\lambda}}
\end{align*}
this gives a lower bound of
\begin{align*}
\left\lceil \frac{m}{4} \right\rceil \cdot t &\geq \frac{m}{4}\left\lfloor \frac{\sqrt{\tilde{Q}}-1}{4}\log\frac{\epsilon^i_0}{20\sqrt{\tilde{Q}}\epsilon} \right\rfloor \\
&\geq \frac{m}{4}\left\lfloor \frac{\sqrt{\tilde{Q}}-1}{4}\log\frac{\epsilon_0}{40\sqrt{\tilde{Q}}\epsilon} \right\rfloor \\
&\geq \frac{m}{4} \frac{\sqrt{\tilde{Q}}-1}{8}\log\frac{\epsilon_0}{40\sqrt{\tilde{Q}}\epsilon} \\
&\geq \frac{m}{4} \frac{3}{40\sqrt{m\lambda}}\log\frac{\epsilon_0\sqrt{m\lambda}}{30\epsilon} \\
&= \frac{3}{160}\sqrt{\frac{m}{\lambda}}\log\frac{\epsilon_0\sqrt{m\lambda}}{30\epsilon} \\
&= \Omega\left( \sqrt{\frac{m}{\lambda}}\log\frac{\epsilon_0\sqrt{m\lambda}}{\epsilon} \right)
\end{align*}

The same argument as was used in the discussion after theorem \ref{thm:RandSBLower} to show the $\Omega(m)$ term of the lower bound can be used here, as the function in that construction was both smooth and strongly convex.

As mentioned above, Lemma \ref{lem:keyrand4} requires that the norm of all query points be bounded by $B$. We argued above that the optimum of $F$ must lie within the $B$-ball around the origin. Even so, we can slightly modify our construction to show that even if the algorithm were allowed to query arbitrarily large points, it still would not be able to optimize $F$ quickly. Define $f'_{i,j}$ through its gradient as:
\[ 
\nabla f'_{i,j}(x) = \begin{cases} \nabla f_{i,j}(x) & \norm{x} \leq B 
\\ \nabla f_{i,j}\left(B\frac{x}{\norm{x}}\right) - \lambda B\normalized{x} + \frac{\lambda}{2}\norm{x}^2 & \norm{x} > B \end{cases}
\]
This new function is continuous, $\gamma$-smooth, and $\lambda$-strongly convex, and it also has the property that querying the function at a point $x$ outside the $B$-ball, it is no more informative than querying at $B\frac{x}{\norm{x}}$. That is, an algorithm that was not allowed to query outside the $B$-ball could simulate the result of such queries. Since that restricted algorithm can't optimize $F'$ well, as proven above, another algorithm which could query at arbitrary points, therefore could not either.
\end{proof}

\subsection{Non-smooth components when $\epsilon$ is large \label{appendix:bigm}}
\begin{theorem}
For any $L,B >0$, any $\frac{10}{\sqrt{m}} < \epsilon < \frac{1}{4}$, and any $m \geq 161$, there exists $m$ functions $f_i$ which are convex and $L$-Lipschitz continuous defined on $\mathcal{X} = \set{x \in \mathbb{R} : \abs{x}\leq B}$ such that for any randomized algorithm $A$ for solving problem \eqref{eq:main} using access to $h_F$, $A$ must make at least $\Omega\left(\frac{L^2B^2}{\epsilon^2}\right)$ queries to $h_F$ in order to find a point $\hat{x}$ such that $\mathbb{E}[F(\hat{x}) - F(x^*)] < \epsilon$. 
\end{theorem}
\begin{proof}
Without loss of generality, we can assume that $L = B = 1$. We construct $m$ functions $f_i$ on $\mathbb{R}^1$ in the following manner: first, sample $p$ from the following distribution
\begin{equation*}
p = \begin{cases}  \frac{1}{2} - 2\epsilon & \text{w.p. } \frac{1}{2} \\ \frac{1}{2} + 2\epsilon &  \text{w.p. } \frac{1}{2}  \end{cases}
\end{equation*}
Then for $i = 1,...,m$, we define
\begin{equation*}
 f_i(x) = \begin{cases} x & \text{w.p. } p \\ -x & \text{w.p. } 1-p \end{cases} 
 \end{equation*}
Consider now the task of optimizing $F(x) = \frac{1}{m}\sum_{i=1}^m f_i(x) = \frac{Yx}{m}$. Clearly, $F$ is optimized at $-\text{sign}(Y)$, and as long as $\abs{Y} > 2m\epsilon$, then any $x$ which is $\epsilon$-suboptimal given $\text{sign}(Y) = +1$ must be at least $3\epsilon$-suboptimal given $\text{sign}(Y) = -1$. Using Chernoff bounds, $\mathbb{P}(\abs{Y} \leq 2m\epsilon) \leq \exp(-\frac{\epsilon^2m}{2}) < \exp(-5)$. Therefore, since the expected suboptimality of an iterate $x$ is at least
\begin{align*}
\mathbb{E}\left[ F(x) - F(x^*) \right] &\geq 3\mathbb{P}\left(\text{sign}(x) \neq -\text{sign}(Y) \middle| \abs{Y} \geq 2m\epsilon \right)\mathbb{P}\left( \abs{Y} \geq 2m\epsilon \right)\cdot \epsilon \\
&> 3(1 - \exp(-5)) \mathbb{P}\left(\text{sign}(x) \neq -\text{sign}(Y) \middle| \abs{Y} \geq 2m\epsilon \right)\cdot \epsilon 
\end{align*}
Therefore, until the algorithm has made enough queries so that 
\[ \mathbb{P}\left(\text{sign}(x) \neq -\text{sign}(Y) \middle| \abs{Y} \geq 2m\epsilon \right) < \frac{1}{3 - 3\exp(-5)} \]
the expected suboptimality is greater than $\epsilon$. By a standard information theoretic result \citep{Agarwal09,LeCam}, achieving that probability of success at predicting the sign of $Y$ implies a comparable level of accuracy at distinguishing between $p = 0.5 + 2\epsilon$ and $p = 0.5 - 2\epsilon$, and that requires at least $\frac{1}{128\epsilon^2}$ queries to $h_F$. 
\end{proof}
It is straightforward to show a lower bound of $\Omega\left(\frac{L^2}{\lambda\epsilon}\right)$ for strongly convex functions using the same reduction by regularization as in the proofs of theorems \ref{thm:DetLscLower} and \ref{thm:RandLscLower}. We also note that this lower bound implies a lower bound of $\Omega(m)$ for smooth functions, whether strongly convex or not. Each function $f_i$ in this construction is linear, and therefore is trivially $0$-smooth. We make the gradient of each function arbitrarily large by multiplying each $f_i$ by a large number. As the multiplier grows, the algorithm need be more and more certain of the sign of $Y$ in order to achieve a expected suboptimality of less than $\epsilon$. Thus for a sufficiently large multiplier, the algorithm must query $\Omega(m)$ functions. We cannot force it to query more than that, of course, since it only needs to query $m$ functions to know the sign of $Y$ with probability 1.

\end{document}